\documentclass[12pt, english, a4paper]{article}
\usepackage[latin9]{inputenc}
\usepackage[T1]{fontenc}
\usepackage{graphicx}
\usepackage{amsthm}
\usepackage{amsfonts}
\usepackage{amsmath}
\usepackage{amssymb}
\usepackage{makecell}
\usepackage{comment}
\usepackage{cases}
\usepackage{mathrsfs}
\usepackage{fullpage}
\usepackage{times}
\usepackage[usenames]{color}
\usepackage[dvipsnames]{xcolor}
\usepackage{hyperref}

\newcommand{\SSS}{\mathfrak{S}}
\newcommand{\BB}{\mathfrak{B}}
\newcommand{\DD}{\mathfrak{D}}
\newcommand{\AAA}{\mathcal{A}}
\newcommand{\id}{\mathrm{id}}
\newcommand{\bs}{\mathbf{s}}
\newcommand{\bp}{\mathbf{p}}
\newcommand{\lh}{\mathcal{LH}}
\newcommand{\var}{\mathrm{var}}
\newcommand{\bkt}[1]{\langle #1 \rangle}
\newcommand{\ol}[1]{\overline{#1}}
\newcommand{\blue}[1]{\textcolor{blue}{#1}}

\newcommand{\ZZ}{\mathbb{Z}}

\DeclareMathOperator{\drp}{drops}
\DeclareMathOperator{\mad}{MAD}
\DeclareMathOperator{\motz}{Motz}
\DeclareMathOperator{\drops}{drops}
\DeclareMathOperator{\dep}{depth}
\DeclareMathOperator{\Pet}{Pet-GP}
\DeclareMathOperator{\depth}{depth}
\DeclareMathOperator{\DescSet}{DescSet}
\DeclareMathOperator{\DescSetD}{DescSet_D}
\DeclareMathOperator{\ExcS}{ExcSet}
\DeclareMathOperator{\exc}{exc}
\DeclareMathOperator{\des}{des}
\DeclareMathOperator{\lexc}{ExcLen}
\DeclareMathOperator{\fix}{FIX}
\DeclareMathOperator{\rc}{RC}
\DeclareMathOperator{\inv}{inv}
\DeclareMathOperator{\SgnDrp}{SgnDrop}
\DeclareMathOperator{\SgnMult}{SgnExcDepDrop}
\DeclareMathOperator{\rd}{canon\_reduc}
\DeclareMathOperator{\asc}{AscSet}
\DeclareMathOperator{\ird}{indx\_canon\_reduc}
\DeclareMathOperator{\nml}{nml}
\DeclareMathOperator{\Negs}{Negs}
\DeclareMathOperator{\nsum}{NegSum}
\DeclareMathOperator{\inva}{inv_A}
\DeclareMathOperator{\invb}{inv_B}
\DeclareMathOperator{\invd}{inv_D}
\DeclareMathOperator{\DescSetB}{DescSet_B}
\DeclareMathOperator{\drpb}{drops_B}
\DeclareMathOperator{\SgnDrpB}{SgnDrop_B}
\DeclareMathOperator{\Sdropd}{SgnDrop_D}
\DeclareMathOperator{\zdrops}{dropz}
\DeclareMathOperator{\iexc}{iexc}

\theoremstyle{plain}
\newtheorem{theorem}{Theorem}
\newtheorem{corollary}[theorem]{Corollary}
\newtheorem{lemma}[theorem]{Lemma}
\newtheorem{proposition}[theorem]{Proposition}
\theoremstyle{definition}
\newtheorem{definition}[theorem]{Definition}
\newtheorem{example}[theorem]{Example}

\newtheorem{remark}[theorem]{Remark}

\usepackage{authblk}
\title{Canonical reduced words and signed descent length enumeration in Coxeter groups}
\author[1]{Umesh Shankar\thanks{\tt{204093001@iitb.ac.in, umeshshankar@outlook.com}}}
\author[2]{Sivaramakrishnan Sivasubramanian\thanks{\tt{krishnan@math.iitb.ac.in}}}
\affil[1,2]{Department of Mathematics, Indian Institute of Technology, Bombay Mumbai 400076, India}

\date{\today}
\begin{document}
\maketitle
\begin{abstract}
Reifegerste and independently, Petersen and Tenner studied a statistic
$\drops$ on permutations in $\mathfrak{S}_n$.
Two other 
studied statistics on $\SSS_n$ are $\dep$ and $\exc$.  
Using descents in {\em canonical reduced words} of elements in $\mathfrak{S}_n$, 
we give an involution $f_A: \mathfrak{S}_n \mapsto \mathfrak{S}_n$ that leads to
a neat formula for the signed trivariate enumerator of 
$\drops, \dep, \exc$ in $\SSS_n$.  This gives 
a simple formula for the signed univariate drops enumerator in $\mathfrak{S}_n$.  
For the type-B Coxeter group $\mathfrak{B}_n$, using similar techniques, 
we show analogous univariate results.  For the type D Coxeter group, we get
analogous but inductive univariate results.  
%Our proof motivates a sign preserving version of two bijections, by 
%Foata and Steingr{\'i}msson respectively in the $\SSS_n$ case.  

Under the famous Foata-Zeilberger bijection $\phi_{FZ}$ which takes 
permutations to restricted Laguerre histories, we show that 
permutations $\pi$ and $f_A(\pi)$ map to the same Motzkin path, but 
have different history components. Using the Foata-Zeilberger 
bijection, we also get a continued fraction for the generating function
enumerating the pair of statistics $\mathrm{drops}$ and $\mathrm{MAD}$.
Graham and Diaconis determined the mean and the variance of  
the Spearman metric of disarray
$D(\pi)$ when one samples $\pi$ from $\mathfrak{S}_n$ at random.
As an application of our %signed enumeration 
results, we get the
mean and variance  of the statistic $\mathrm{drops}(\pi)$ when we sample 
$\pi$ from $\mathcal{A}_n$ at random.
\end{abstract}
%\textbf{\small{}Keyword:}{\small{} }{\let\thefootnote\relax\footnotetext{The author is supported by the National Board for Higher Mathematics, India.}}{\let\thefootnote\relax\footnotetext{2020 \textit{Mathematics Subject Classification}. Primary ; 
%Secondary , .}}

\section{Introduction}

Let $\SSS_n$ be the set of permutations on $[n] = \{1,2,\ldots,n\}$.  While studying
an optimization problem motivated by a sorting algorithm called {\it 
straight selection sort}, Petersen and Tenner in 
\cite{petersen-tenner-depth} came up with the following statistic 
they call {\it depth} for elements of $\SSS_n$ and $\BB_n$, the
hyperoctahedral group.  For the permutation 
group $\SSS_n$, they gave a simple alternate expression for 
depth.  Define the {\it excedance set} of 
$\pi = \pi_1, \pi_2, \ldots, \pi_n \in \SSS_n$ as 
$\ExcS(\pi) = \{ i \in [n-1]: \pi_i > i\}$.  Define 
$\depth(\pi) = \sum_{i \in \ExcS(\pi)} (\pi_i - i)$.  

 Bagno, Biagioli, Novick and Woo in 
\cite{bagno-biagioli-novick-woo-depth-coxeter-groups} extended the
notion of depth to all Coxeter groups.
A closely related statistic on $\SSS_n$ was defined by Graham and Diaconis
in \cite{graham-diaconis-spearman}.  For 
$\pi = \pi_1, \pi_2, \ldots, \pi_n \in \SSS_n$, 
Graham and Diaconis defined $s(\pi) = \sum_{i=1}^n |\pi_i-i|$
to be its Spearman disarray measure.  Knuth in 
\cite[Problem 5.1.1.28]{knuth-taocp2} 
calls $s(\pi)$ as the ``total 
displacement'' of $\pi$.  It is clear that for all $\pi \in \SSS_n$, 
$s(\pi)= 2 \cdot \depth(\pi)$.  Though a generating function 
for $\depth$ does not seem to have a simple form, Petersen and 
Guay-Paquet 
\cite{petersen-guay-paquet-displacement} expressed the 
generating function for $\depth$ as a continued fraction.  

Define a very similar 
statistic {\it descent drops}, denoted $\drp(\pi)$ as follows.  Firstly,
for $\pi = \pi_1, \pi_2, \ldots, \pi_n$,
define $\DescSet(\pi) = \{i \in [n-1]: \pi_i > \pi_{i+1} \}$.  Define 
\begin{equation}
  \label{eqn:drop_defn}
  \drp(\pi) = \sum_{i \in \DescSet(\pi)} (\pi_i - \pi_{i+1}).
\end{equation}

Both these statistics have been studied by Reifegerste 
in \cite{reifegerste-bi-incr}.
For $\pi \in \SSS_n$, define $\exc(\pi) = |\ExcS(\pi)|$ and 
$\des(\pi) = |\DescSet(\pi)|$.  
Reifegerste in \cite{reifegerste-bi-incr} and Petersen and Tenner 
independently in \cite{petersen-tenner-depth}, showed that the distribution of the 
ordered pairs $(\depth(\pi), \exc(\pi))$ and $(\drp(\pi), \des(\pi))$
coincide when summed over $\SSS_n$.  Their result 
\cite[Proposition 1.1]{reifegerste-bi-incr} and 
\cite[Theorem 1.3]{petersen-tenner-depth} is as follows.

\begin{theorem}[Reifegerste, Petersen and Tenner]
  \label{thm:biv_distrib}
  When $n \geq 1$, we have 
  $$\sum_{\pi \in \SSS_n} q^{\depth(\pi)}t^{\exc(\pi)} = 
  \sum_{\pi \in \SSS_n} q^{\drp(\pi)}t^{\des(\pi)}$$
\end{theorem}

This result in particular shows that when  $n \geq 1$,
$\sum_{\pi \in \SSS_n} q^{\drp(\pi)} = 
\sum_{\pi \in \SSS_n} q^{\depth(\pi)}$.  
Sivasubramanian in \cite{siva-exc-det}
showed that enumerating $\depth(\pi)$ along with its sign over $\SSS_n$
gave interesting results.  In \cite{siva-exc-det}, $\depth(\pi)$ was 
termed {\it excedance-length} of $\pi$ and denoted $\lexc(\pi)$.  For 
$\pi \in \SSS_n$, define $\inv(\pi) = |\{1 \leq i < j \leq n: \pi_i > \pi_j \}|$.
He proved the following  in \cite[Theorem 9]{siva-exc-det}.

\begin{theorem}[Sivasubramanian]
  \label{thm:sgn_exc}
  When $n \geq 1$, we have 
  $$\sum_{\pi \in \SSS_n} (-1)^{\inv(\pi)} q^{\depth(\pi)} =
  \sum_{\pi \in \SSS_n} (-1)^{\inv(\pi)} q^{\lexc(\pi)} = (1-q)^{n-1}.$$
\end{theorem}

Thus, a natural question is the signed enumeration of the
$\drp$  in $\SSS_n$.  Here, we show that enumerating
signed drops in $\SSS_n$, gives the same result. Even more surprising 
is that the signed joint distribution of the three statistics, $\exc$, 
$\depth$ and $\drops$, factors nicely.  
For $n \geq 1$, define $\SgnMult_n(q) = \sum_{\pi \in \SSS_n} 
(-1)^{\inv(\pi)} t^{\exc(\pi)}p^{\depth(\pi)}q^{\drops(\pi)}$.
Our first main result of this paper is the following.

\begin{theorem}
    \label{thm: both-stat}
    For $n\ge 1$, we have
%    \begin{equation*}
        $\displaystyle \SgnMult_n(p,q)=(1-tpq)^{n-1}.$	
%        \end{equation*}
\end{theorem}

This result generalises the results of Mantaci {\cite[Proposition 1]{mantaci-signed}} 
and Theorem \ref{thm:sgn_exc} while giving us the following new corollary. 
For $n\ge 1$, define $\SgnDrp_n(q)=\sum_{\pi \in \SSS_n} (-1)^{\inv(\pi)}q^{\drops(\pi)}$.
\begin{corollary}
  \label{thm:sgn_drp_univ} 
  When $n \geq 1$, we have $\SgnDrp_n(q) = (1-q)^{n-1}$.
\end{corollary}

% Theorem   \ref{thm:sgn_drp_univ}  and the remark that follows 
% Theorem  \ref{thm:biv_distrib} show that 
% the pair of statistics $\drp$ and $\depth$ have the same 
% signed {\bf and} unsigned distribution when summed over 
% elements of $\SSS_n$, though 
% the bivariate generating functions $f_n(q,t) = \sum_{\pi \in \SSS_n} 
% q^{\inv(\pi)} t^{\depth(\pi) }$ 
% and
% $g_n(q,t) = \sum_{\pi \in \SSS_n} q^{\inv(\pi)} t^{\drp(\pi)}$ are 
% different when $n \geq 4$. 
% Thus, the pair of statistics $\drp$ and $\depth$ are 
% equidistributed over the alternating group $\AAA_n$ and over $\SSS_n-\AAA_n$. 
% It can be checked that both the proof of Reifegerste of Theorem 
% \ref{thm:biv_distrib}, which uses a bijection $f:\SSS_n \mapsto \SSS_n$ of 
% Foata \cite[Theorem 10.2.3]{lothaire-combin-words} and the proof of 
% Petersen and Tenner, which uses a bijection $g: \SSS_n \mapsto \SSS_n$ 
% of Steingr\'{i}msson \cite{steingrimsson-indexed-perms}, 
% cannot be extended to $\AAA_n$ and $\SSS_n-\AAA_n$ because 
% these bijections neither reverse nor preserve signs of the permutations.

Our proof of Theorem  \ref{thm: both-stat} uses a weight preserving,
 sign reversing involution $f_A$ based on the idea of {\it canonical 
reduced words}. In Section \ref{sec:canon_red_words}, we 
cover the relevant material on reduced words in type A and 
type B Coxeter groups.
In Section \ref{sec:invol}, we define our involution and 
in Section \ref{sec:proof_typea}, we show that our involution preserves
the $\drops$ statistic.
The famous Foata-Zeilberger bijection from 
\cite{foaza-denert-indeed} uses the idea of restricted Laguerre history 
and maps permutations in $\SSS_n$ bijectively to weighted Motzkin paths of
length $n$. For an introduction to Motzkin paths, we suggest the paper
\cite{Motzkin-paths-oste-jeugt}  by Oste and Van der Jeugt 
and for heights in Motzkin paths, we refer the reader to 
\cite{height-two-types-Motzkin-brennan-knopfmacher} by Brennan
and Knopfmacher.
 In Section \ref{sec:Motzkin-paths}, we show that our involution 
$f_A$ preserves the shape 
of the Motzkin path under the Foata-Zeilberger bijection. 
Using this observation, we finish the proof of Theorem \ref{thm: both-stat} and 
also show that under the Foata-Zeilberger bijection, the fixed points of 
our involution $f_A$ map to Motzkin paths whose height is at most one.

The connection to Motzkin paths helps us get a continued fraction
expression for the pair of statistics $(\depth, \inv)$.  This form
for the continued fraction was observed and conjectured by 
Petersen in his email to Zeilberger \cite{doron-petersen-email}.   
In Theorem \ref{thm: genfun-dep-inv} we give our proof of this conjecture.  
Using %the Flajolet theory connecting continued fractions and Motzkin paths, and 
a bijection of Clarke, Steingr{\'i}mmson and Zeng \cite{clarke-stein-zeng}, 
we get a continued fraction expression for the pair of statistics $(\drp, \mad)$
over $\SSS_n$ (see Section \ref{sec:cont-frac} for details).

%\blue{Does our definition coincide with the definition of Bagno et al?}
In $\SSS_n$, as the definition of $\drp$ depends on {\it descents},
we consider a straightforward generalization to type-B and type-D  Coxeter groups. 
In  Section \ref{sec:type-B}, we prove a type B counterpart (see Theorem 
\ref{thm:sgn_drp_typeB}) of Corollary \ref{thm:sgn_drp_univ}, where we 
enumerate type-B drops with signs, in $\BB_n$ using very similar 
ideas on reduced words.  Our type D counterpart is Theorem 
\ref{thm-typed} and is given in Section \ref{sec:type-d}.  Our 
type D result is not proved using canonical reduced words but 
inductively.

% For the type A case, we get a little more from our reduced word based 
% sign reversing involution.   
% The famous Foata-Zeilberger bijection from 
% \cite{foaza-denert-indeed} uses the idea of restricted Laguerre history 
% and maps permutations in $\SSS_n$ to weighted Motzkin paths with
% length $n$. For a quick introduction to Motzkin paths, we suggest the paper
% \cite{Motzkin-paths-oste-jeugt}  by Oste and Van der Jeugt 
% and for heights in Motzkin paths, we refer the reader to 
% \cite{height-two-types-Motzkin-brennan-knopfmacher} by Brennan
% and Knopfmacher.
% In Section \ref{sec:Motzkin-paths}, 
% we show that under the Foata-Zeilberger bijection, the fixed points of 
% our involution map to Motzkin paths whose height is at most one. This will show 
% that our involution preserves both the $\iexc$, $\drops$ and $\dep$ statistics. 
% Define $\SgnMult_n(p,q)=\sum_{\pi \in \SSS_n} (-1)^{\inva(\pi)}t^{\exc(\pi)}p^{\depth(\pi)}q^{\drops(\pi)}$ for $n\ge 1$.

% \begin{theorem}
%     \label{thm: both-stat}
%     For $n\ge 1$, we have
%     \begin{equation*}
%         \SgnMult_n(p,q)=(1-tpq)^{n-1}
%         \end{equation*}
% \end{theorem}

% Setting $p=1$ and $q=1$ respectively gives us back Theorems $\ref{thm:sgn_drp_univ}$ and $\ref{thm:sgn_exc}$. 
%\blue{We have a continued fraction expression - check Zeilberger's 
%page 
%https://sites.math.rutgers.edu/~zeilberg/mamarim/mamarimhtml/noga12yleFeedback.pdf
%}

Lastly, we give two application of our enumeration results.   
Among other results, Graham and Diaconis in \cite[Theorem 1]{graham-diaconis-spearman} 
determined the mean and variance of the  {\bf Spearman-disarray 
metric} statistic  when permutations are sampled randomly from
$\SSS_n$.  In Subsection \ref{subsec:asymp_norm}, we determine the mean 
and the variance of the $\drp$ statistic when permutations are randomly 
sampled from $\AAA_n$. 
In Subsection \ref{subsec:matching}, we show 
that our reduced word based involution gives a matching in the Bruhat order 
on the type A and type B Weyl groups.   Our matching is different from 
the one given by Jones \cite{jones-matching-bruhat} for the type A case.

\section{Canonical reduced words}
\label{sec:canon_red_words}

We briefly cover the relevant material on canonical reduced words in 
type A and type B Coxeter groups needed for our proof.  We point
the reader to Humphreys 
\cite{humphreys-cox}, Stembridge \cite{stembridge-combin-aspects} 
and Garsia \cite{garsia-saga} as good 
background and references for this topic.

Let $(W,S)$ be a Coxeter system. As each element $w\in W$ can be written as 
a product of generators from $S$, its {\it length} is the minimal number 
of generators required in any such product. 
The type-A Coxeter group $\SSS_n$ is generated by 
$S = \{s_1, s_2, \ldots, s_{n-1} \}$ where $s_i$ is the transposition $(i,i+1)$.
The hyperoctahedral group $\BB_n$ is the type-B Coxeter group consisting
of permutations of the set $[\pm n] = \{-n, -(n-1), \ldots, -1,1,2,\ldots,n \}$ 
satisfying $\pi(-i) = -\pi(i)$ for $1 \leq i \leq n$.  It is known that $\BB_n$ is 
generated by $S = \{s_0,s_1,\ldots,s_{n-1} \}$ where $s_0$ 
only switches the sign of the first element.
When $1 \leq i < n$,  $s_i$ is an identical transposition as in the $\SSS_n$ case.  
Descents in Coxeter groups are related to reduced words as follows.  Let $w \in W$ have length $\ell(w)$. Then, $s \in \DescSet(w)$ iff $\ell(ws) < \ell(w)$ 
(see Bj{\"o}rner and Brenti \cite{bjorner-brenti}).  
We will write ``$w$ has a descent at position $i$'' to mean $s_i \in \DescSet(w)$.

For both type $A$ and type $B$ Coxeter groups, recall that 
multiplying an element $w = w_1,w_2,\ldots, w_n \in W$ (given in 
one-line notation), by $s_i$ for $i=1,\ldots,n-1$ on the right 
swaps %the entries in positions $i$ and $i+1$ (i.e. swaps 
$w_i$ and $w_{i+1}$. For type-B Coxeter groups, multiplying $w$ by 
$s_0$ on the right flips the sign of $w_1$.  
In both $\SSS_n$ and $\BB_n$, we denote the empty word by $1$.   
To get the reduced word for $\pi$, we start from the identity 
permutation and apply $s_i$ on positions get to $\pi$.

\subsection{Type-A Coxeter Groups}  
\label{subsec:type-a}

Given $\pi \in \SSS_n$,
we start from the identity permutation and mimic an iterative 
``reverse sorting" process to get to $\pi$ and use this to 
define the canonical reduced word $\rd(\pi)$.   We have $(n-1)$
stages which we number from $1$ to $n-1$ in increasing order. 
For $1 \leq i \leq n-1$, at the $i$-th stage, we ensure that 
the $i$-th element from 
the right agrees with $\pi$.  At subsequent stages, these 
elements which are set in their correct place do not move.
It is clear that after performing this for $(n-1)$ stages, 
we will get $\pi$.  Further, it is clear that the factorization 
obtained for $\pi$ has length $\inv(\pi)$ and so we term
the concatenated string of factors of this algorithm as 
$\rd(\pi)$.  Note that we 
have started correction from the rightmost element. Alternatively, one could
have started correction from the leftmost element (this 
is done by Garsia, see \cite[Theorem 1.1.1]{garsia-saga}).

\begin{example}
\label{eg:type-a}
In one-line notation, let 
$\pi = 4,1,5,2,3 \in \SSS_5$.  
We start with $\pi_0 = \id = 1,2,3,4,5$.   Since $3$ is
the last letter of $\pi$, we move the $3$ to the last position by
applying $r_4 = s_3s_4$ (from left to right)  to $\pi_0$.  This gives 
us $\pi_1 = 1,2,4,5,3$.  Since $2$ appears in the penultimate position 
of $\pi$, we apply $r_3 = s_2s_3$ to $\pi_1$ to get $\pi_2 = 1,4,5,2,3$.  Since
$5$ is in its correct position, we apply $r_2 = 1$ to $\pi_2$ to get
$\pi_3$ (in this case $\pi_3 = \pi_2$).  Finally, we apply $r_1 = s_1$ to
$\pi_3$ to get $\pi_4 = \pi$. Thus, we get the canonical 
factorisation of $\pi$ as $\rd(\pi) = [r_4][r_3][r_2][r_1]$, 
that is $\rd(\pi) = [s_3 s_4] [s_2 s_3] [1]  [s_1]$.  Square brackets 
are only given above for added clarity.
\end{example}

To see the set in which $r_i$ belongs, let 
$W^{\bkt{i} } = \{1, s_i, s_{i-1} s_i, \ldots, s_1 s_2 \cdots s_i \}$.  
By this reverse sorting process, it is easy to see that
for all $\pi \in \SSS_n$, we will get $\rd(\pi) = [r_{n-1}][r_{n-1}]
\cdots[r_1]$ where $r_i \in W^{\bkt{i}}$.

\begin{remark}
\label{rem:len-type-a}
In $\SSS_n$, it is easy to see that $W^{\bkt{1}}$ has only elements 
of length $0$ or $1$ in it.
\end{remark}

\subsection{Type-B Coxeter Groups}  
\label{subsec:type-b}

The basic idea is the same in this case as well.  The only difference is 
that now, to get elements with a negative sign, we need to move it to
the first place, apply $s_0$ and move it back to its correct position.
As opposed to the $\SSS_n$ case, we will have $n$ stages as at the last 
stage, the first element could
possibly get negated.  Thus if we use the same notation as in Subsection
\ref{subsec:type-a}, we will get factors $[r_n][r_{n-1}]\cdots[r_1]$.
To see where the factors $r_i$ belong,  for $1 \leq i \leq n$, define 
$$
W^{\bkt{i}}  = \{1, s_{i-1}, s_{i-2}s_{i-1}, \ldots, s_0s_1\cdots s_{i-1}, 
	s_1s_0s_1\cdots s_{i-1}, \ldots, s_{i-1} \cdots s_1s_0s_1 \cdots s_{i-1}\}.
$$

Note that $W^{\bkt{1}} = \{1, s_0 \}$.

\begin{remark}
\label{rem:len-type-b}
As in the type-A case, note that $W^{\bkt{1}}$ has only elements 
of length $0$ or $1$ in it.  
\end{remark}

We start  from the identity permutation 
$\id \in \BB_n$ and analogously perform operations to move elements
of $\pi = \pi_1,\pi_2,\ldots, \pi_n \in \BB_n$ to the right such that 
for $1 \leq i \leq n$, after $i$ iterations, the $i$ rightmost 
elements are exactly as they appear in $\pi$.  We will get 
$\rd(\pi) = [r_n][r_{n-1}] \cdots [r_2][r_1]$, where $r_i \in W^{\bkt{i}}$ and
as before, square brackets are only given for clarity.
For example, if in one-line notation, we have
$\sigma = 4,1,\ol{5},2,\ol{3}$, then one can check that applying 
$r_5 = s_2s_1s_0s_1s_2s_3s_4$  to $\id$ results in $1,2,4,5,\ol{3}$.  Thus, 
the rightmost element is in its correct place.  Proceeding in a similar manner,
one can check that 
$\rd(\sigma) = [s_2s_1s_0s_1s_2 s_3s_4] [s_2s_3] [s_2s_1s_0s_1s_2] [s_1] [1].$

\section{Two sign reversing involutions and some properties}
\label{sec:invol}

Let $W$ be either $\SSS_n$ or $\BB_n$.
We look at the indices of the canonical
reduced word for each $w \in W$.  This is a sequence of integers of length $\ell(w)$.
For a finite sequence $x = x_1, x_2, \ldots, x_t$ of integers, define 
its ascent set to be  $\asc(x) = \{ i \in [t-1]: x_i < x_{i+1}\}$.  For $w \in W$, 
recall $\rd(w)$ is the canonical reduced word of $w$.   As the generators
are denoted by $s_{r}$ where $r$ is an integer, 
let $\ird(w)$ 
be the sequence of these indices of the generators that appear in $\rd(w)$. 
Let $\asc( \ird(w))$ be the ascent set of the word $\ird(w)$.

\begin{remark}
\label{rem:no-asc-between-wi}
In both the type A and the type B case,  the following is easy to see.
Since all elements of $W_i$ always end with $s_i$ and start with 
some $s_j$ with $j < i$, there cannot be an ascent in $\ird(w)$
between the last letter of $r_i$ and the first letter of $r_j$ when $i > j$. 
%\blue{There is no ascent between an elements of $W^{i}$ and $W^{i+1}$.
%Is this what we want to say?}
\end{remark}

By Remark \ref{rem:no-asc-between-wi}, in $\ird(w)$, ascents 
can appear only between two letters, where both letters are 
elements of $W^{\bkt{i}}$.
Further, all ascents will be between consecutive numbers.
The word {\it consecutive} is used with respect to the total order on
the generators.  
That is, all ascents in $\ird(w) = r_1, r_2, \ldots, r_{\ell}$ 
will be at indices 
$j$ for some $j < \ell$, where $r_j = a$ and $r_{j+1} = a+1$.
Consider Example \ref{eg:type-a} and recall 
$\pi = 4,1,5,2,3 \in \SSS_5$.  We clearly have
its indices $\ird(\pi) = 3,4,2,3,1$ and $\asc(\ird(\pi)) = \{1,3\}$.

\subsection{Type-A Coxeter groups}
\label{subsec:invol_typeA}
We define the involution $f_A$ for the type-A case first.  
For $w \in \SSS_n$, if $\ird(w)$ has ascents, then 
let $i \in \asc( \ird(w) )$ be its largest element.  That is,
consider the last occurrence of an ascent. 
From the discussion in the first paragraph of Section \ref{sec:invol}, 
such an ascent will be of the form $a,a+1$ (as it arises from 
$s_a s_{a+1}$).  In $w$, change $s_a s_{a+1}$ to $s_a s_{a+1} s_a$.  As 
this modified expression is a product of generators, we get an 
element $w' \in W$ and define 
$f_A(w) = w'$.  If $\ird(w)$ does not have descents, then define $f_A(w) = w$.

Alternatively, we could have defined the map $f_A$ as follows.  Let $w \in \SSS_n$ 
and let its canonical reduced word 
be $\rd(w) = [r_{n-1}] [r_{n-2}] \cdots [r_2] [r_1]$ with
$r_i \in W^{\bkt{i} }$.
Let $\ell(r_i)$ be 
%the length of $r_i$.  That is, $\ell(r_i)$ is 
the number of $s_j$'s in $r_i$.  
Consider the {\it smallest} $i$ such that $\ell(r_i) \geq 2$.  
If $\ell(r_i) \geq 2$,  by Remark \ref{rem:len-type-a}, $i \geq 2$ 
and the string $r_i$ will end with $s_{i-1} s_i$.  The map 
$f_A$ converts $r_i$ to $r_is_{i-1}$ %as strings 
and does not change any other $r_k$'s.  As $f_A(w)$ is a product of generator
elements, after a possible reduction (using $s_i^2 = 1$), we 
will get an element $w' \in W$  and define $f_A(w) = w'$.
If $\pi \in \SSS_n$ is such that $\ell(r_i) \leq 1$ for all $i$, 
then define $f_A(\sigma) = \sigma$.
We give an example of the bijection $f_A$ below.  

\begin{example}
From Example \ref{eg:type-a}, when $\pi = 4,1,5,2,3$,  
$\rd(\pi) = [s_3 s_4] [s_2s_3] [1]  [s_1]$.   Thus, %under $f_A$, we have 
$f_A(\pi) = [s_3 s_4] [s_2s_3\textcolor{blue}{s_2}] [1] [s_1] = [s_3 s_4] [s_2s_3] [s_2] 
[s_1]$ and so $f_A(\pi) = 5,1,4,2,3$.  
%One can see that 
%$\pi$ and $\sigma = f(\pi)$ have $\drp(\pi) = \drp(\sigma) = 5$ and further that
%$\inva(\pi) \not\equiv \inva(\sigma)$ (mod 2).
\end{example}

\subsection{Type-B Coxeter groups}
For $\sigma \in \BB_n$, let 
$\rd(\sigma) = [r_n][r_{n-1}] \cdots [r_1]$.  
We differentiate between the $r_i$'s using the following definition.
For $1 < i \leq n$, define $r_i \in W^{\bkt{i}}$ as being $\nml$ (short 
for {\sl near-maximal-length}) if $r_i = u_i$ or if  $r_i = v_i$ where   
%\begin{center}
  $u_i = s_{i-1}s_{i-2}\cdots s_0s_1\cdots s_{i-1} \in W^{\bkt{i}}$ 
%\hspace{2 mm}   
and %\hspace{2 mm}
$v_i = s_{i-2}\cdots s_0s_1\cdots s_{i-1} \in W^{\bkt{i}}$.  
%\end{center}

We define an involution $g_B$ on $\BB_n$ in two parts as follows.  
First, consider the case when $\sigma \in \BB_n$
is such that $\rd(\sigma)$ has a factor $r_i$ which is $\nml$. 
In this case, let $i$ be the largest index with $r_i$ being $\nml$ 
(that is, $r_i$ is the leftmost $\nml$ factor).  We change $u_i$ to $v_i$ 
and vice versa.  That is, 
if $r_i = u_i$, then $\rd(\sigma) = [r_n]\cdots[u_i] \cdots [r_1]$. 
In this case, 
define $g_B(\sigma) = [r_n] \cdots[v_i] \cdots [r_1]$.  If $r_i = v_i$, then,
$\rd(\sigma) = [r_n]\cdots[v_i] \cdots [r_1]$.  In this case, define
$g_B(\sigma) = [r_n] \cdots[u_i] \cdots [r_1]$.  

If $\sigma\in \BB_n$, does not have $\nml$ factors in $\rd(\sigma)$, 
we proceed as done in the type-A involution $f_A$.  
In this case, let $t$ be the smallest index with 
$\ell(r_t) \geq 2$.  As seen before, we can show $t\geq 2$.  Further, $r_t$ will 
end with $s_{t-1}$.  Define 
$g_B(\sigma) = [r_n] \cdots [r_{t+1}][r_t s_{t-2}][r_{t-1}]\cdots[r_1]$.  
Lastly, if $\sigma \in \BB_n$ is such
that $\rd(\sigma)$ has neither a $\nml$ factor, nor a factor
$r_i$ with $\ell(r_i) \geq 2$, then define $g_B(\sigma) = \sigma$.
We give two examples of our bijection $g_B$ below.  

\begin{example}
  \label{eg:typeB_one}
Let
$\sigma = 3,1,\ol{5},2,\ol{4} \in \BB_5$.  We have 
$\rd(\sigma) = [r_5] [r_4] [r_3] [r_2] [r_1]$ where
$r_5 = s_3s_2s_1s_0s_1s_2 s_3s_4$, $r_4= s_2s_3$, $r_3= s_2s_1s_0s_1s_2$, 
$r_2= s_1$
and $r_1 = 1$.  It is easy to see that $\rd(\sigma)$ has two $\nml$ factors 
viz $r_5 = u_5 $ and $r_3 = v_3$.  Among them, the 
leftmost $\nml$ factor is $r_5$.
We replace $r_5 = u_5$ by $r_5 = v_5$.  Thus, we get 
$g_B(\sigma) = [\textcolor{blue}{s_4s_3s_2s_1s_0s_1s_2 s_3s_4}] [s_2s_3]  
[s_2s_1s_0s_1s_2]  [s_1][1]$,
that is $g_B(\sigma) = 3,1,\ol{4},2,\ol{5}$.
\end{example}

\begin{example}
Let $\sigma = \ol{6},1,5,3,4,\ol{2} \in \BB_6$.  Then, $\rd(\sigma) = [r_6][r_5][r_4][r_3][r_2][r_1]$
where $r_6 = s_2s_1s_0s_1s_2s_3s_4s_5$, $r_5=s_3s_4$, $r_4 = s_2s_3$, $r_3 = 1$, $r_2 = s_1$ and 
$r_1= s_0$.
Since there is no $\nml$ factor,
we use the type-A style involution and change $r_4 = s_2s_3$ to $s_2s_3s_2$ to get
$g_B(\sigma) = [r_6][r_5][\blue{s_2s_3s_2}][r_3][r_2][r_1]$.  Hence, we
get $g_B(\sigma) = \ol{5},1,6,3,4,\ol{2}$.
\end{example}

\subsection{An intermediate sequence of elements of $W$}
\label{subsec:interm_elem}
In this subsection, for $W$ being both the type-A and type-B Coxeter groups, 
given $w \in W$,  we define a sequence of elements starting with 
the respective identity elements $\id$ and ending at $w$.  
This sequence will be used in the proof of 
Lemmas  \ref{lem:drp_preserve_typeA} and \ref{lem:drp_preserve_typeB}.

Let $W$ have $n$ generators and let $w \in W$ have $\rd(w) = [r_n][r_{n-1}]\cdots[r_1]$.
Define $w_{n+1} = \id$ and for $1 \leq i \leq n$, define $w_i = \prod_{j=n}^i [r_j]$.
Since $r_j$ is a product of generators, $w_i \in W$.  Clearly, $w_1 = w$ and we 
think of the sequence $w_j$ as $j$ decreases from $j=n+1$ to $j=1$ as a sequence of elements
starting from $\id$ and getting closer to $w$.  It is not necessary for elements 
$w_j$ of this sequence to be distinct (see Example \ref{eg:interm_elems}).  We 
call elements $w_j$ as {\bf the sequence of intermediate elements} of $w$.  It is clear from
the definition of $w_i$ that $\rd(w_i) = \prod_{j=n}^i [r_j] \prod_{j=i-1}^1 
[1] = \prod_{j=n}^i [r_j]$.

\begin{example}
  \label{eg:interm_elems}
Let $w = 3,1,\ol{5},2,\ol{4} \in \BB_5$.  Then, as seen in
Example \ref{eg:typeB_one}, 
$\rd(\sigma) = [r_5] [r_4] [r_3] [r_2] [r_1]$ where
$r_5 = s_3s_2s_1s_0s_1s_2 s_3s_4$, $r_4= s_2s_3$, $r_3= s_2s_1s_0s_1s_2$, $r_2= s_1$
and $r_1 = 1$.  In this case, we have $w_6 = \id$, $w_5 = [r_5] = 1,2,3,5,\ol{4}$,
$w_4 = [r_5][r_4] = 1,3,5,2,\ol{4}$, $w_3 = 1,3,\ol{5},2,4$ and  $w_2 = w_1 = 3,1,\ol{5},2,\ol{4} = w$.
\end{example}

\subsection{A few lemmas}

We prove a few lemmas about the map $f_A$ defined above and give
without proof similar lemmas about the map $g_B$.

\begin{lemma}
  \label{lem:canonical_preserve}
  If $w \in \SSS_n$, then $f_A(w)$ perhaps after using the relation 
$s_i^2=1$ is the canonical reduced
  word of some $w' \in \SSS_n$.
\end{lemma}
\begin{proof}
  Let $T \subseteq \SSS_n$ be those elements $w \in W$ such that 
$\ird(w)$ has no ascents.  As $T$ is the set of fixed points of $f_A$,
if $w \in T$, then $f_A(w) = w$ and so the lemma trivially follows.  Thus,
assume that $w \in W - T$.  
Recall that if
$\rd(w) = [r_n] \cdot [r_{n-1}] \cdots [r_2] \cdot [r_1]$, then we modify the last
index $j$ such that $\ell(r_j) \geq 2$.  As given in the alternate definition of
the map $f_A$, $j \geq 2$ and so $j-1 \geq 1$.  Further, $\ell(r_{j-1})$ takes only two 
values $0$ or $1$ and both values give rise to canonical reduced words of some elements
$u,w \in W$.  

We claim that $f_A$ swaps the two possibilities for $\ell(r_{j-1})$.  To see this,
suppose $\ell(r_{j-1}) = 1$.  Then $r_{j-1} = s_{i-1}$ and
since $f_A$ changes $r_j$ to $r_j s_{i-1}$, after using the relation $s_{i-1}^2 = 1$,
we get that $r_{j-1} = 1$, i.e $\ell(r_{j-1}) = 0$.  Similarly, if $\ell(r_{j-1}) = 0$, 
then as $f_A$ changes $r_j$ to $r_j s_{j-1}$, we will have $r_{j-1} = s_{j-1}$.  Hence
$\ell(r_{j-1}) = 1$.  Thus, $f_A$ switches these
two possibilities for $r_{j-1}$ and does not change other $r_j$'s.  Hence
$f_A(w)$ is also a canonically reduced word, completing the proof.
\end{proof}

In the next two lemmas, we claim that $f_A$ is an involution and that it
reverses the sign on non fixed-points.

\begin{lemma}
\label{lem:invol}
The map $f_A$ defined above is an involution.  i.e. $f_A(f_A(w)) = w$ for all
$w \in W$.
\end{lemma}
\begin{proof}
  Let $T$ be the set defined in the proof of Lemma \ref{lem:canonical_preserve}.
  The lemma is trivial if $w \in T$.  If $w \in W - T$, then,
  from the proof of Lemma \ref{lem:canonical_preserve}  it follows
  that if $j$ is the smallest index such that $\ell(r_j) \geq 2$, 
  then $f_A$  flips the two possible values 0 and 1 of 
  $\ell(r_{j-1})$.  Hence $f_A(f_A(w)) = w$, completing the proof.
\end{proof}

\begin{lemma}
  \label{lem:sign_rev}
Let $W = \SSS_n$ and $T$ be the fixed point set of $f_A$ defined in Lemma 
\ref{lem:canonical_preserve}.
For $w \in W - T$, $\inv(w) \not \equiv \inv(f_A(w))$ (mod 2).
\end{lemma}
\begin{proof}
If $\rd(w)$ is a product of $k$ generators, then by definition,
$\inv(w) \equiv k$ (mod 2).  By construction, $f_A(w)$ is a product of $k+1$ 
or $k-1$ transpositions (as $\ell(r_{j-1})$ changes parity).
By the Deletion property 
(see \cite[Pg 17]{bjorner-brenti}), $f_A(w)$ has length 
$p$ where $p \not \equiv k$ 
(mod 2).  Thus the signs of $w$ and $f_A(w)$ are different.
\end{proof}

An analogous result is true for the involution $g_B$ and since the proof is not 
very different, we merely state the next result and omit its proof. 
%\blue{\bf Check this.}

\begin{lemma}
  \label{lem:sign_rev_b}
Let $W$ be the type-B Coxeter group %with rank $n$ 
and let $w \in W$.  Then, $g_B(w)$ perhaps after using the 
relation $s_i^2 = 1$ will give a reduced decomposition
of some $w' \in W$.  Further, the map $g_B$ is an involution and reverses
sign on the set of its non fixed-points.
\end{lemma}

The following lemma will be needed for our results in Section
\ref{sec:Motzkin-paths}.  Since it involves the  map $f_A$, we present
it here.

\begin{lemma} 
\label{lemma:transposition-lem}
If $f_A(\pi) \not= \pi$, then there exists a transposition $(a,b)$ such that 
$(a,b)f_A(\pi)=\pi$. Further,  if $i$ is the index where the canonical reduced 
word decompositions of $\pi$ and $f_A(\pi)$ differ and if $w_{i+1}$ 
is the $(i+1)$-st intermediate permutation of $\pi$, then $a,b$ are 
the letters at the positions $i,i+1$ in $w_{i+1}$ respectively. 
\end{lemma}
\begin{proof}
Let $\pi=[r_{n-1}]\dots[r_1]$ and $f_A(\pi)=[t_{n-1}]\dots[t_1]$ be the 
canonical reduced word decompositions of $\pi$ and $f_A(\pi)$ respectively. 
Let $i$ be the only index where they differ. Then, we have $r_i=s_it_i$. 
Furthermore, by our involution, $l(r_j)\le 1$ for all $j\le i$. 
Let $2\le i'<i$ be the largest index such that $r_{i'}=1$. Then, for each 
$i'<j<i$, $r_j=t_j=s_j$. Let $w_j,z_j$, for $1\le j\le n$, be the 
intermediate permutations of $\pi,f_A(\pi)$ respectively. 
Since $r_j=t_j$ for all $j\ge i+1$, we have $w_j=z_j$ for $j\ge i+1$. 

Without loss of generality, let $r_i=1$ and $t_i=s_i$.  Since 
$w_{i+1}=z_{i+1}$, suppose that 
$w_{i+1}=x_1\dots x_{i-1}bax_{i+2}\dots x_n$. This implies 
$w_{i+1}.r_i=w_i=x_1\dots x_{i-1}bax_{i+2}\dots x_n$. Since $r_j=s_j$ for 
$i'<j<i$, $w_{i'}= w_i\displaystyle \prod_{i'<j<i} s_j $ and we get
$w_{i'}=x_1\dots x_{i'}bx_{i'+1}\dots x_{i-1}ax_{i+2}\dots x_n$.

Similarly, since $t_i=s_i$, we have $z_i=x_1\dots x_{i-1}abx_{i+2}\dots x_n$.
$z_{i'}=z_i\displaystyle \prod_{i'<j<i} s_j$ and so
$z_{i'}=x_1\dots x_{i'}ax_{i'+1}\dots x_{i-1}bx_{i+2}\dots x_n$.  By our
choice of $i'$, $r_{i'}=t_{i'}=1$ and for $1\le j < i'$, the
transpositions $r_j,t_j$ act only on the subword $x_1\dots x_{i'}$ for
$z_{i'}$ and $w_{i'}$ respectively. Therefore, $\pi$ and $f_A(\pi)$ differ
only by the transposition $(a,b)$ and $a,b$ are the letters at 
positions $i,i+1$ respectively in $w_{i+1}$.
\end{proof}

\section{Proof of Theorem \ref{thm: both-stat}}
\label{sec:proof_typea}

Recalling the definition of intermediate
permutations from Subsection \ref{subsec:interm_elem},
we first prove the following lemma.  
\begin{lemma}
\label{lem:increase_typeA}
Let $W = \SSS_n$ %be the rank $n-1$, type-A Coxeter group 
and let $w \in W$ have 
intermediate elements $w_i$ for $i \in [n]$.  For $i \in [n]$, let 
the one line notation of $w_i$ be
$w_i = a_1,a_2,\ldots,a_n$.  Then, $a_1 < a_2 \cdots < a_i$.  
\end{lemma}
\begin{proof}
The proof is by reverse induction on $i$ with the base case being 
when $i=n$.  As $\pi_n = \id$, the statement of the lemma is 
satisfied.  Let the result be true when $i=k$ and let $\rd(w) = 
[r_{n-1}] [r_{n-2}]  [r_2] [r_1]$.   We will prove that
the claim holds when $i=k-1$.  When $i=k$, if  
$w_k = a_1, a_2, \ldots, a_n$, then by induction,
$a_1 < a_2 < \cdots < a_k$.  As $w_{k-1} = w_k [r_{k-1}]$ and since 
$r_{k-1} \in W^{k-1}$, we see that some element $a_j$ where $j \in [k]$ 
comes to position $k$ after moving the block of elements 
$[a_{j+1},\ldots, a_k]$ one position
to the left.  Thus, if the one-line notation of $w_{k-1}$ is 
$w_{k-1} = a_1', a_2', \ldots, a_n'$, then 
this movement clearly does not alter the relative order among the elements 
$a_1', a_2', \ldots, a_{k-1}'$.  
Thus $a_1' < a_2' < \cdots < a_{k-1}'$, completing the proof.
\end{proof}

%As an illustration of the earlier lemma, note that when $w = ( ) \in \SSS_n$, then 
%$w_1 = $ and so ...

\begin{lemma}
  \label{lem:drp_preserve_typeA}
  Let $W = \SSS_n$.   %be the type-A Coxeter group.  
  For $w \in W$, $\drp(w) = \drp(f_A(w))$.
\end{lemma}
\begin{proof}
	If $f_A(w) = w$, the lemma is trivial.  Thus, assume $f_A(w) = z$, with $z \not= w$.
  Let $\rd(w) = [r_{n-1}] [r_{n-2}]  [r_2] [r_1]$ and 
  $\rd(z) = [q_{n-1}] [q_{n-2}]  [q_2] [q_1]$.  
Let $t$ be the smallest index such that $\ell(r_t) \geq 2$.  We have seen that
$t \geq 2$ and that $r_i = q_i$ for all indices $i \in [n-1]$, $i \not= t-1$.
For $1 \leq i \leq n$, recall the intermediate permutations $w_i$ of $w$ 
%defined as $w_n = \id$ and when $1 \leq i < n$, then 
%$w_i = w_{i+1} \cdot r_i$.
%  Similarly, let 
and $z_i$ %be the intermediate permutations 
of $z$.
%with  $z_n = \id$, $z_i = z_{i+1} \cdot q_i$ with $z_1 = z$.   
%From the definition of $f_A$, 
We have seen that
either $\ell(r_{t-1}) = 1, \ell(q_{t-1}) = 0$ or vice versa.
Without loss of generality, assume that $\ell(r_{t-1}) = 1$.  Since
  $r_{t-1} \in W^{t-1}$,  $r_{t-1} = s_{t-1}$
  and $q_{t-1} = 1$.  Thus, we have $w_{t-1} = w_t \cdot s_{t-1}$
  and $z_{t-1} = z_t$.  This is depicted pictorially in Figure 
  \ref{fig:eg_intermediate_perms}.
  
\begin{figure}[h]
  \centerline{\includegraphics[scale=0.55]{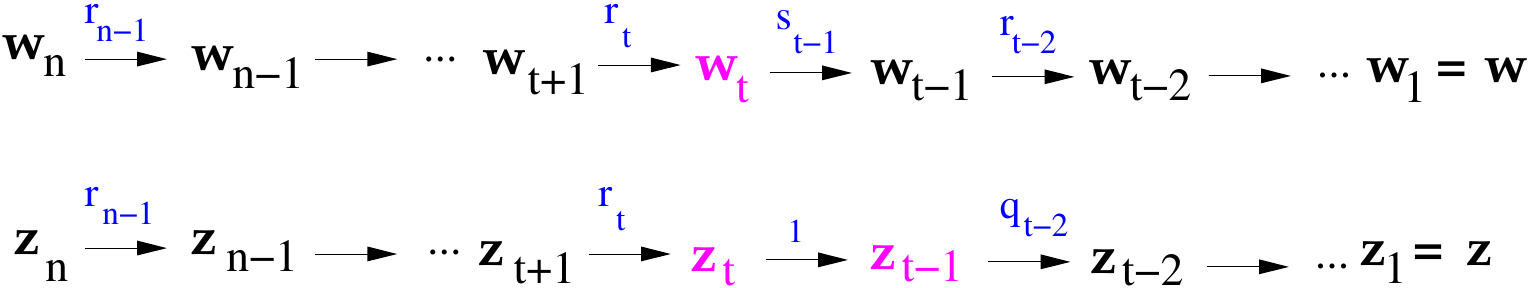}}
  \caption{Intermediate permutations.}
  \label{fig:eg_intermediate_perms}
\end{figure}

We prove the stronger assertion that $\drp(w_i) = \drp(z_i)$
for all $i \in [n]$ by reverse induction on $i$. 
The claim is true when $i=n$ as $w_n = z_n = \id$.  Further, as $r_i = q_i$ 
when $n-1 \geq i \geq t$, we clearly have $w_i = z_i$ when $n \geq i \geq t$.
Hence $\drp(w_i) = \drp(z_i)$ when $t \leq i \leq n$.

  %From the description of $f_A$, we clearly have permutations $\pi_{t-1}$ and $\psi_{t-1}$ 
  %(i.e. $t-1 \geq 1$).
  %Further, we have $\psi_{t-1} = \psi_t$ and $\pi_{t-1} = \pi_t \cdot s_{t-1}$.  
Since the canonical reduced word for $w_t$ (and $z_t$) ends in $s_t$, position
$t \in \DescSet(w_t)$ and $t \in \DescSet(z_t)$.
In one-line notation, let $w_t = z_t = z_{t-1} = x_1, x_2, \ldots, x_n$ and 
let $w_{t-1} = y_1, y_2, \ldots, y_n$.  Since $w_{t-1} = w_t \cdot s_{t-1}$, 
we have  
\begin{equation}
  \label{eqn:trio}
  x_t = y_{t-1}, x_{t-1} = y_t \mbox{ and } x_{t+1} = y_{t+1}.
\end{equation}

By Lemma \ref{lem:increase_typeA} applied to $w_t$, we get  $x_1 < x_2 < \cdots < x_t$.
Since a reduced word for $w_{t-1}$ ends in $s_{t-1}$, position $t-1$ is a descent for 
$w_{t-1}$ and thus we have $y_{t-1} > y_t$.  Using (\ref{eqn:trio}) we rephrase this
as $x_t > x_{t-1}$.  Further, as $\ell(r_t) \geq 2$,
and as $w_{t+1} \cdot r_t = w_t$, the element $x_{t+1}$ occurred at an 
index smaller than $t-1$ in $w_{t+1}$, and moved to the $t$-th position in $w_t$.
During this move, clearly, a block of $\ell(r_t)$ elements got shifted by one 
position to the left and $x_{t+1}$ took the position $t+1$.  
Applying Lemma  \ref{lem:increase_typeA} to $w_{t+1}$, we get

\begin{equation}
  \label{eqn:trio_reln}
x_{t+1} < x_{t-1} < x_t.
\end{equation}

%e{The above equation is the key to the proof for $\drp(w_t) = \drp(z_t)$. }
Note that (\ref{eqn:trio_reln}) is about the relative order of elements of $w_t$.
Further, as a reduced word for $w_t$ ends in $s_t$, position $t$ will be a descent of $w_t$.
Thus, $\drp(w_t) = (x_t - x_{t+1}) + D$ where $D = \sum_{ c \in \DescSet(w_t) - \{t\}} 
(x_c - x_{c+1})$.  Since $w_t \cdot s_{t-1} = w_{t-1}$, by Lemma  
\ref{lem:canonical_preserve}, $\rd(w_{t-1}) = [r_{n-1}]\cdot[r_{n-2}] \cdots [r_t][r_{t-1}]$.
Since $\ell(r_t) \geq 2$, $r_t$ ends with $s_{t-1}s_t$.  Since $r_{t-1} = s_{t-1}$, a 
reduced word for $w_{t-1}$ ends with $s_{t-1} s_t s_{t-1}$.  Thus, some reduced word
for $w_{t-1}$ has the form $f s_{t-1} s_t s_{t-1}$ for some $f$ written as a product 
of generators.  Using the Coxeter relations, this reduced word can be changed to 
$f s_t s_{t-1} s_t$.  Thus, there are two reduced words for $w_{t-1}$, one which ends
with $s_t$ and another which ends with $s_{t-1}$.  Thus $t-1, t \in \DescSet(w_{t-1})$.
Clearly, the other descents of $w_{t-1}$ are identical to that of $w_t$.  Thus.
$w_{t-1}$ has descents at positions $t-1, t$ and at $k \in \DescSet(w_t) - \{t\}$.
Thus 
\begin{eqnarray*}
\drp(w_{t-1}) & = & (y_{t-1} - y_t) + (y_t - y_{t+1}) + D  
					 =  (y_{t-1} - y_{t+1}) + D \\
					 & = & (x_t - x_{t+1}) + D = \drp(w_t) = \drp(z_{t-1})
\end{eqnarray*}
where we have used (\ref{eqn:trio}) in the last line.  

We next show that $\drp(w_{t-2}) = \drp(z_{t-2})$.  When $k > 2$, 
the same argument will show that $\drp(w_{t-k}) = \drp(z_{t-k})$.  
Since $t$ is the smallest 
index with $\ell(r_t) \geq 2$, whenever $k \geq 2$, we will have $\ell(r_{t-k}) = 0/1$.
If $\ell(r_{t-2}) = 0$, then $r_{t-2} = 1$ and there is nothing to prove.  
Thus, we assume that $r_{t-2} = q_{t-2} = s_{t-2}$.  

%Recall
%$\pi_{t-1} = (y_1, y_2, \ldots, y_n)$ and $\psi_{t-1} = (x_1, x_2, \ldots, x_n)$.
%Let $\pi_{t-2} = (y_1',y_2', \ldots, y_n')$ and $\psi_{t-2} = (x_1', x_2', \ldots, x_n' )$.
%By the above observation, $x_{t-2}' = x_{t-1}, x_{t-1}' = x_{t-2}$ and 

Lemma \ref{lem:increase_typeA} applied to $z_t$ gives $x_{t-2} < x_{t-1}$.
As $z_{t-2} = z_t \cdot s_{t-2}$, these are precisely the elements swapped 
in $z_t$ to get $z_{t-2}$.  Thus, 
$\DescSet(z_{t-2}) = \DescSet(z_t) \cup \{t-2 \}$. 
Similarly, we claim that
$\DescSet(w_{t-2}) = \DescSet(w_t) \cup \{ t-2\}$.  
%\blue{\bf Does this need substantiation}.  
To see this, recall that as $w_t = x_1,x_2,\ldots,x_n$, 
by  (\ref{eqn:trio_reln}), we have $x_{t+1} < x_{t-1} < x_t$.  
If $w_{t-2} = y_1',y_2', \ldots, y_n'$, 
as $w_{t-2} = w_t \cdot s_{t-1} \cdot s_{t-2}$, we get 

\begin{equation}
  \label{eqn:quadro}
y'_{t+1} = x_{t+1},
y_{t-2}' = x_t, y_{t-1}' = x_{t-2}, \mbox{ and }
y'_t = x_{t-1}.  
\end{equation}

Further, if $z_{t-2} = x_1', x_2', \ldots, x_n'$, as 
$z_{t-2} = z_t \cdot s_{t-2}$, we have 
$x'_{t-2} = x_{t-1}$, $x'_{t-1} = x_{t-2}$ and $x'_t = x_t$.  
In $w_t$, since position $t$ is a descent, we have $x_t > x_{t+1}$ and let 
$D = \sum_{q \in \DescSet(w_t) - \{t\} } (x_q - x_{q+1})$.
Recall that $\DescSet(z_{t-2}) = \DescSet(z_t) \cup \{t-2 \}$
and that $\DescSet(w_{t-2}) = \DescSet(w_t) \cup \{ t-2\}$.  Thus,
%\blue{\bf Check this and next para carefully}
\begin{eqnarray*}
  \drp(z_{t-2}) & = & \drp(z_t) + (x'_{t-2} - x'_{t-1}) 
  	= \drp(z_t) + (x_{t-1} - x_{t-2} ) \\
  & = & D + (x'_t - x'_{t+1}) + (x_{t-1} - x_{t-2})  =  D + (x_t - x_{t+1}) + (x_{t-1} - x_{t-2}) \\
  & = &  D + (x_t - x_{t-2}) + (x_{t-1} - x_{t+1}) 
   =  D + (y'_{t-2} - y'_{t-1}) + (y'_t - y'_{t+1}) \\
& = & \drp(w_{t-2})
\end{eqnarray*}

Where, the last line follows from (\ref{eqn:quadro}).
An equivalent argument which we omit shows that $\drp(w_{t-k}) = \drp(z_{t-k})$ for $k > 2$,
completing the proof.
\end{proof}

\subsection{The FZ map and height in Motzkin paths}
\label{sec:Motzkin-paths}
Before we state our main result of this section, we need some definitions.
For a permutation $\pi\in \SSS_n$, we call an index $1\le i\le n$ to be
\begin{itemize}
    \item a {\it cyclic peak} denoted as $i\in \mathrm{CPk}(\pi)$ if $\pi^{-1}(i)<i>\pi(i)$,
    \item a {\it cyclic valley} denoted as $i\in \mathrm{CVal}(\pi)$ if $\pi^{-1}(i)>i<\pi(i)$,
    \item a {\it cyclic double ascent} denoted as $i\in \mathrm{Cda}(\pi)$ if $\pi^{-1}(i)<i<\pi(i)$,
    \item a {\it cyclic double descent} denoted as $i\in \mathrm{Cdd}(\pi)$ if $\pi^{-1}(i)>i>\pi(i)$,
    \item a {\it fixed point} denoted as $i\in \mathrm{Fix}(\pi)$ if $\pi(i)=i$.
    \end{itemize}
\begin{definition}
A $2$-Motzkin path of length $n$ is a word $w=s_1s_2\dots s_n$ on 
the alphabet $\{ \mathrm{N,S,E,dE} \}$ (abbreviations for 
North, South, East and dotted East respectively)
such that for each $i$, the height or level $h_i(w)$ of the 
$i^{th}$ step of $w$, defined by 
$$h_i(w)=\#\{j\vert j<i, s_j=\mathrm{N} \}-\#\{j\vert j<i,s_j=\mathrm{S} \}$$ 
is non-negative and is equal to zero when $i=n$.
\end{definition}

\begin{definition} 
A restricted Laguerre history of length $n$ is an ordered pair 
$(\bs,\bp)$, 
where $\bs = (s_1,s_2,\ldots,s_n)$ is a $2$-Motzkin path of length 
$n$
 and 
$\bp=(p_1,\dots,p_n)\in \mathbb{N}^n$ satisfies $0\le p_i\le h_{i}(\textbf{s})$ 
when $s_i=\mathrm{N},\mathrm{E}$ and $0\le p_i \le h_i(\textbf{s})-1$ when 
$s_i=\mathrm{S},\mathrm{dE}$. We denote the set of restricted Laguerre 
histories of length $n$ by $\lh_n^*$.
\end{definition}

We refer to the $2$-Motzkin path $\textbf{s}$ of a restricted Laguerre 
history $(\textbf{s},\textbf{p})$, as its shape, 
denoted by $\mathrm{sh}(\textbf{s},\textbf{p})$, and the 
vector $\textbf{p}=(p_1,\dots,p_n)$ as the labels of the 
restricted Laguerre history, denoted by 
$\mathrm{h}(\textbf{s},\textbf{p})$. 

Next, we recall the famous Foata-Zeilberger bijection \cite{foaza-denert-indeed} 
between $\SSS_n$ and  $\lh_n^*$ using cyclic statistics.  
Given $\sigma \in \SSS_n$, define 
$\phi_{FZ}(\sigma)=(\textbf{s},\textbf{p}) \in \lh_n^*$ where for $i=1,2,\dots,n$,
\[
s_i=
\left\{
\begin{array}{ll}
\mathrm{N}  & \mbox{if } i\in \mathrm{Cval}(\sigma), \\
\mathrm{S}  & \mbox{if } i\in \mathrm{Cpk}(\sigma), \\
\mathrm{E}  & \mbox{if } i\in \mathrm{Cda}(\sigma)\cup \mathrm{Fix}(\sigma), \\
\mathrm{dE} & \mbox{if } i\in \mathrm{Cdd}(\sigma),
\end{array}
\right.\]
with $p_i=\mathrm{nest}_i(\sigma)$ where 
$\mathrm{nest}_i(\sigma)= | \{j\vert j<i<\sigma(i)<\sigma(j) \text{ or } \sigma(j)<\sigma(i)\le i<j \}|.$
Define $\mathrm{nest}(\sigma)=\sum_{i=1}^n \mathrm{nest}_i(\sigma)$.
To determine the cyclic statistics of $f_A(\pi) \in \SSS_n$, we need 
a few lemmas.

\begin{lemma}
\label{lemma: greater-than-lemma}
For $1\le i\le n$, if $w_i$ are the intermediate permutations 
of $\pi\in \SSS_n$, then $w_i(j)\ge j$ for $1\le j\le i$.
\end{lemma}

\begin{proof}
This follows from the proof of Lemma \ref{lem:increase_typeA}.
\end{proof}

We use Lemma \ref{lemma: greater-than-lemma} to bound the $a,b$ 
that appear in Lemma $\ref{lemma:transposition-lem}$.
\begin{lemma}
Let $i$ be the index where $\pi, f_A(\pi)$ differ in their 
canonical reduced word decomposition and let $a,b$ be as in 
Lemma $\ref{lemma:transposition-lem}$, then $a\ge i+1$ and $b\ge i+2$.  
\end{lemma}
\begin{proof} 
By Lemma \ref{lemma:transposition-lem}, we know that $a,b$ are in 
positions $i,i+1$ respectively of $w_{i+1}$. Without loss of generality, 
assume that $a$ is in position $i$ and $b$ is in position $i+1$. 
We know that $w_{i+1}=w_{i+2}.r_{i+1}$ and $l(r_{i+1})\ge 2$ 
(by the definition of $f_A$). Since $l(r_{i+1})\ge 2$, $a$ had 
to have been in the position $i+1$ and $b$ had to be in position 
$i+2$ of $w_{i+2}$ because $a$ is moved to position $i$ and $b$ 
moved to position $i+1$ in $w_{i+1}$. Therefore, by Lemma 
\ref{lemma: greater-than-lemma} for $w_{i+2}$, we get 
$a\ge i+1$ and $b\ge i+2$. This completes the proof.
\end{proof}

With these lemmas in hand, we now prove our 
main result of this section.  If $i\in \mathrm{CPk}(\pi)$
(respectively $\mathrm{CVal}(\pi),\mathrm{Cda}(\pi)\cup \mathrm{Fix}(\pi),\mathrm{Cdd}(\pi)$), 
then we will show that $i\in \mathrm{CPk}(f_A(\pi))$ (respectively $\mathrm{CVal}(f_A(\pi)), \mathrm{Cda}(f_A(\pi))\cup 
\mathrm{Fix}(f_A(\pi)),\mathrm{Cdd}(f_A(\pi))$). This will prove that the shape of 
the associated 
restricted Laguerre histories are the same.

\begin{proposition}
\label{thm: shape} 
Let $\pi \in \SSS_n$, and $f_A$ denote our involution defined
in Subsection \ref{subsec:invol_typeA}.  Then,
$$\mathrm{sh}(\phi_{FZ} (\pi))=\mathrm{sh}(\phi_{FZ}(f_A(\pi))).$$
\end{proposition}
\begin{proof} 
We show that if $i\in \mathrm{CPk}(\pi)$ (respectively 
$\mathrm{CVal}(\pi),\mathrm{Cda}(\pi)\cup \mathrm{Fix}(\pi),\mathrm{Cdd}(\pi)$), then 
$i\in \mathrm{CPk}(f_A(\pi))$ (respectively $\mathrm{CVal}(f_A(\pi)),\mathrm{Cda}(f_A(\pi))
\cup \mathrm{Fix}(f_A(\pi)),\mathrm{Cdd}(f_A(\pi))$).  Let $a,b$ be such that $(a,b)f_A(\pi)
=\pi$. Without loss of generality, assume that $a$ appears to the left of $b$ in the one 
line notation of $\pi$. Thus, we have 
\begin{eqnarray*}
\pi & = & \pi(1),\dots,a,\dots,\pi(i),b,\pi(i+2),\dots,\pi(n) \mbox{ and } \\
f_A(\pi) & = & \pi(1),\dots,b,\dots,\pi(i),a,\pi(i+2),\dots,\pi(n).
\end{eqnarray*}
If $\pi^{-1}(a)=i'$, then $\pi$ and $f_A(\pi)$ 
are different only at the indices $i',i+1$. Therefore, we have the following simple relations. If $j\not \in \{i',i+1\}$, then $\pi(j)=f_A(\pi)(j)$ and 
if $j\not \in \{a,b\}$, then $\pi^{-1}(j)=f_A(\pi)^{-1}(j)$.
To prove the claim, we only need to check for the 
following three cases.
%the claim for indices $j$  such that one of the three occurs.  
\begin{enumerate}
\item {\bf Case 1, when $j=i+1$ or $j=i'$ :}
Let $j=i'$, then $\pi(j)=a$. Since $a\ge i+1$, by Lemma 
\ref{lemma: greater-than-lemma} and $j=i'\le i$, we have $j<\pi(j)=a$. 
If $\pi(j)=a$, then $f_A(\pi)(j)=b$.  We have $j<b=f_A(\pi)(j)$. 
Since $i'<a$ and $i'<b$, $j=i'$ is not equal to either $a$ or $b$ 
and therefore, $\pi^{-1}(j)=f_A(\pi)^{-1}(j)$. These produce the 
same letters under $\phi_{FZ}$.
    
Let $j=i+1$, then $\pi(j)=b$. Since $j< b$, it can be $a$. If not, we 
have $\pi^{-1}(j)=f_A(\pi)^{-1}(j)$.  If $a=i+1=j$, then 
$f_A(\pi)^{-1}(a)=a=f_A(\pi)(i+1)=a$ and $i'=\pi^{-1}(a)<i+1<\pi(i+1)=b$, 
then they both produce the same letters under $\phi_{FZ}$ as 
cyclic double ascents and fixed points both produce the same letter.

\item {\bf Case 2, when $j=a$ or $j=b$: }
If $j=a$, then $j\ge i+1$ and we further split into two cases 
$a=i+1$ and $a\ge i+2$. The former case was covered earlier in Case 1. 
For the latter case, clearly, the letters $\pi(j)$ will coincide 
for $\pi$ 
and $f_A(\pi)$ when $j\ge i+2$. Therefore, we have 
$\pi(a)=f_A(\pi)(a)$. 
Also, $\pi^{-1}(a)= i' <i+1 \le a$ and $f_A(\pi)^{-1}(a)=i+1< 
i+2\le a$. Thus, they produce the same letters under the map $\phi_{FZ}$. 

Suppose $j=b$, then $j\ge i+2$ and the letters $\pi(j)$ 
coincide for $\pi$ and $f_A(\pi)$ for $j\ge i+2$.  Further, 
$\pi^{-1}(b)= i+1 < b$ and $f_A(\pi)^{-1}(b)=i'< i+2\le b$. 
Thus, they produce the same letters under the map $\phi_{FZ}$.

\item {\bf Case 3, when $j=\pi(a)$ or $j=\pi(b)$ (i.e. $j=f_A(\pi)(a)$ or $j=f_A(\pi)(b)$):} 
Since $a\ge i+1$, $\pi(a)=f_A(\pi)(a)$ and thus we have 
$\pi^{-1}(j)=f_A(\pi)^{-1}(j)=a$. If $j$ is neither $i+1$ nor $i'$, 
then, $\pi(j)=f_A(\pi)(j)$. These $j$ produce the same letters 
under $\phi_{FZ}$. If we have $j=i'$, we then have $a> i' < a$ 
and $f_A(\pi)^{-1}(i')=a > i' < f_A(\pi)(i')=b$. 
Therefore, they produce the same letters.
    
If we assume $j=i+1$, either $a=i+1$ or $a > i+1 <a$. The case 
when $a=i+1$ has been covered earlier in Case 1. 
If $a>i+1$, then $a > i+1<f_A(\pi)(i+1)=a$. They both have cyclic 
valleys at $i+1$ and therefore, produce the same letters under 
$\phi_{FZ}$.

Since $b \ge i+2$, $\pi(b)=f_A(\pi)(b)$ we have $\pi^{-1}(j)=
f_A(\pi)^{-1}(j)=b$. If $j$ is not $i+1$ or $i'$, then 
$\pi(j)=f_A(\pi)(j)$.  These $j$ produce the same letters 
under $\phi_{FZ}$. If we have $j=i'$, we have $\pi^{-1}(i')=b> i' 
< \pi(i')=a$ and $f_A(\pi)^{-1}(i')=b > i' < f_A(\pi)(i')=b$. Therefore, 
they produce the same letters. If we have $j=i+1$, then 
$b=\pi^{-1}(i+1) > i+1 < \pi(i+1)=b$ $a \ge i+1<f_A(\pi)(i+1)=a$. 
However, the case when $j=a=i+1$ has been covered 
in Case 1. They both have cyclic valleys at $i+1$ and 
therefore, produce the same letters under $\phi_{FZ}$.
    %Since $a\ge i+1$, $\pi(a)=f_A(\pi)(a)=j$. Now, if $j$ is not $a$ or $b$, then $\pi(j)=f_A(\pi)(j)$ and therefore, we are done as they will produce the same letters. If $j=a$, then $j$ is fixed but this is impossible as $b$ occupies position $i+1$ and $a$ lies to the left of it while being larger than $i+1$. Now, if $j=b$, then $a=i+1$. Also, $f_A(\pi)^{-1}(b)=i+1$ and since, $j=b\ge i+1$, $\pi(j)=f_A(\pi)(j)$ and therefore, produce the same letters.   
\end{enumerate}
 The rest of the indices have the same values and positions in $\pi, f_A(\pi)$, by Lemma \ref{lemma:transposition-lem}. This completes the proof.
\end{proof}

Petersen and Guay-Paquet in \cite{petersen-guay-paquet-displacement} 
defined a map $\phi_{\Pet}$ from $\SSS_n$ to the set $\motz_n$ of Motzkin paths 
of length $n$ that is closely related to the Foata-Zeilberger bijection. 
In a Motzkin path 
we do not differentiate between the $\mathrm E$ and $\mathrm {dE}$ steps
unlike a $2$-Motzkin path.
Let $\phi_{\Pet}$ be the map described by Petersen and Guay-Paquet.
Petersen and Guay-Paquet show that $\phi_{\Pet}(\pi)=\mathrm{sh}(\phi_{FZ}(\pi))$ 
For $\pi \in \SSS_n$, they show that $\phi_{\Pet}(\pi)$ is not 
injective but is surjective onto $\motz_n$. 
They used it to get a continued fraction expression for the generating
function of the $\dep$ statistic 
using the shape of the Motzkin path corresponding to each permutation.

%As mentioned above, the map $\pi_{\Pet}$ is surjective and 
To determine the number of pre-images a Motzkin path has, 
Petersen and Guay-Paquet defined the {\it weight} of step $p_i$ as
\[
  \omega_i = \begin{cases}
    h_i &\text{if $p_i = N$ or $p_i = S$,} \\
    2h_i + 1 &\text{if $p_i = E$,}
  \end{cases}
\]
and the weight of a path $p = p_1 \cdots p_n$ to be the product 
$\displaystyle \omega(p) = \omega_1 \cdots \omega_n.$
They showed (see \cite[Proposition 4]{petersen-guay-paquet-displacement})
that the weight of a Motzkin path $p$ 
equals the number of permutations in its preimage $\phi^{-1}(p)$.

\begin{proposition}[Petersen and Guay-Paquet]
\label{prop:weight}
Let $p\in \motz_n$. Then 
$$\omega(p)=\vert\{ \pi\in \SSS_n : \phi_{\Pet}(w)=p \}\vert$$    
\end{proposition}

The following is an easy corollary of the previous proposition.

\begin{corollary}
For $p \in \motz_n$, if $p$ has maximum height strictly more than 
1, then, under the map $\phi_{\Pet}$, there are even number of 
permutations $\pi \in \SSS_n$ with 
$\phi_{\Pet}(\pi) = p$.
%hose maximum height is strictly greater than $1$. 
If $p$ has maximum height at most 1, there are odd number of 
permutations $\pi \in \SSS_n$ with 
$\phi_{\Pet}(\pi) = p$.
\end{corollary}

Next, we give a bijection between the set of Motzkin paths 
of length $n$ with maximum height at most $1$ and  subsets of $[n]$ 
with even cardinality. This will imply that the set of  Motzkin paths 
of length $n$ with maximum height at most $1$ has cardinality  $2^{n-1}$.
\begin{proposition}
There is a bijection between the number of Motzkin paths 
of length $n$ whose height is at most $1$ and the number of 
subsets of $[n]$ whose cardinality is even.
\end{proposition}
\begin{proof}
Let $S\subseteq [n]$ be such that $|S|$ is even and let $S = \{s_1 \dots s_{|S|}\}_<$ be
$S$ with its elements written in ascending order.   For $1\le i\le\frac{|S|}{2}$, 
declare the $s_{2i-1}$-{th} step of the Motzkin path to be $`\mathrm N'$, the 
$s_{2i}$-th step of the Motzkin path to be $`\mathrm S'$ and the rest of steps 
to be $`\mathrm{E}'$.  This process is clearly reversible and gives
us a bijection.   
\end{proof}
Combining all the results above, we conclude that there is a 
bijection between the fixed points of $f_A$ and Motzkin 
paths whose maximum height is at most $1$.

\begin{proposition}
The fixed points of the involution $f_A$ get mapped bijectively, under $\phi_{\Pet}$, 
to Motzkin paths whose maximum height is at most  $1$.
\end{proposition}
\begin{proof}We know that the fixed points of the involution $f_A$ are 
permutations whose reduced words $r_i$ satisfy $\ell(r_i)\le 1$. 
Clearly, there are $2^{n-1}$ such permutations. By Proposition \ref{thm: shape}, 
if the cardinality of the set of permutations that map to a particular Motzkin 
path is odd, then it contains at least one permutation that is fixed by $f_A$. There 
are exactly $2^{n-1}$ such paths and exactly $2^{n-1}$ fixed points. Therefore, 
every fixed point must map to a different Motzkin path whose maximum 
height is at most  $1$.
\end{proof}

We, now, finish the proof of Theorem \ref{thm: both-stat}.
\begin{proof}[Proof of Theorem \ref{thm: both-stat}]
Since the involution $f_A$ preserves the shape of the $2$-Motzkin path under the
Foata-Zeilberger bijection, we have $\mathrm{area}(\phi_{FZ}(\pi))
=\mathrm{area}(\phi_{FZ}(f_A(\pi))$.
Thus $\depth(\pi)=\depth(f_A(\pi))$ because the area under the Motzkin path is
the $\depth$ of the permutation. 
Since the shape is preserved, so is the number of $\mathrm{N}, \mathrm{dE}$ steps. The sum of the
number of $\mathrm{N},\mathrm{dE}$ is exactly the number of $\iexc$ ($\exc$ of the inverse of the permutation) of the permutation, and this is also preserved under
the involution. In Lemma \ref{lem:drp_preserve_typeA}, we showed that the map preserves $\drops$ as well.
Denote by $\fix(f_A)$ the permutations fixed under the involution $f_A$. 
From these three observations, we have
%\begin{align*}
\begin{equation*}
  \sum_{\pi \in \SSS_n}(-1)^{\inv(\pi)}t^{\iexc(\pi)}p^{\dep(\pi)}q^{\drops(\pi)} 
  = \sum_{\pi \in \fix(f_A)}(-1)^{\inv(\pi)}t^{\iexc(\pi)}p^{\dep(\pi)}q^{\drops(\pi)}
  \end{equation*}
%\end{align*}  
The permutations in $\fix(f_A)$ are the ones whose canonical reduced word is of the form $s_{i_1}s_{i_2} \cdots s_{i_k}$ where $n > i_1 > i_2
  > \cdots > i_k \geq 1$. Clearly, such permutations are in bijection with subsets $L$ of 
$[n-1]$ and thus $|T_{n-1}| = 2^{n-1}$.
For each $s_{i_j}$, the $\iexc$ increases by $1$. Since these are adjacent transpositions, the $\depth$ and $\drops$
increases by $1$. Therefore, if $\pi \in \fix(f_A)$, then $\depth(\pi)=\drops(\pi)=\iexc(\pi)=\inv(\pi)$. Therefore,
\begin{align*}
      \sum_{\pi \in \fix(f_A)}(-1)^{\inv(\pi)}t^{\iexc(\pi)}p^{\dep(\pi)}q^{\drops\pi)}=(1-tpq)^{n-1}    
\end{align*} 
Let $\rc$ be the reverse-complement bijection. It is easy to check that $\rc$ is sign preserving and $$(\iexc(\pi),\depth(\pi),\drops(\pi)=(\exc(\rc(\pi)),\depth(\rc(\pi)),\drops(\rc(\pi))).$$
Therefore, 
\begin{align*}
    &\sum_{\pi \in \SSS_n}(-1)^{\inv(\pi)}t^{\iexc(\pi)}p^{\dep(\pi)}q^{\drops(\pi)} \\
     & = \! \! \sum_{\pi \in \SSS_n}(-1)^{\inv(\rc(\pi))}t^{\iexc(\rc(\pi))}p^{\dep(\rc(\pi))}q^{\drops(\rc(\pi))}
    \! \!=\! \!\sum_{\pi \in \SSS_n}(-1)^{\inv(\pi)}t^{\exc(\pi)}p^{\dep(\pi)}q^{\drops(\pi)}
\end{align*}  
This finishes the proof.
\end{proof}

\section{Continued fractions for $(\dep,\inv)$ and $(\drops,\mad)$}
\label{sec:cont-frac}

In this section, we calculate a continued fraction expansion for the 
pair $(\drops, \mad)$.
The $\mad$ statistic was defined and shown to be a Mahonian statistic by 
Clarke, Steingr{\'i}mmson and Zeng \cite{clarke-stein-zeng}.  Any $\pi\in \SSS_n$ 
can be uniquely decomposed into \emph{descent blocks} (maximal descending subwords). 
Denote the first and last letter of each block $B$ of length at least two by 
$c(B), o(B)$ respectively. The \emph{right embracing number} of $\pi(i)$ is the 
number of descent blocks strictly to the right of the block 
containing $\pi(i)$ and for which $c(B) > \pi(i) > o(B)$. The sum of all right
embracing numbers is denoted by $\mathrm{Res}(\pi)$. Then, 
$\mad(\pi):=\drops(\pi)+\mathrm{Res}(\pi)$.

We will get a continued fraction expansion for the pair $(\dep,\inv)$ 
via the Foata-Zeilberger bijection and the theory of continued fractions 
due to Flajolet \cite{cont-frac-flaj}.  
Clarke, Steingr{\'i}mmson and Zeng showed bijectively that the pairs 
$(\drops, \mad)$ and $(\dep,\inv)$ are equidistributed over $\SSS_n$.  
Thus, we can recast our results using this pair.  We get a J-type 
continued fraction expansion where the denominators have linear and 
quadratic terms. For some background on continued fractions, we refer 
the reader to the book by Goulden and Jackson \cite{goulden-jackson}.
Define the $(\dep,\inv)$- enumerating polynomial over $\SSS_n$ as 
$$F_n(x,q)=\sum_{w\in \SSS_n} x^{\dep(w)}q^{\inv(w)}.$$
Our result of this section is the following.  It was conjectured by 
Petersen in his email to Zeilberger (see \cite{doron-petersen-email}).

\begin{theorem} 
\label{thm: genfun-dep-inv}
     The ordinary generating function of $F_n(x, q)$ has the following Jacobi
continued fraction expansion:
$$ \sum_{n\geq 0}F_n(x,q)t^n={1\over\displaystyle 1-c_0t- {\strut
b_1t^2\over\displaystyle 1-c_1t- {\strut
b_2t^2\over\displaystyle {}
{\strut \ddots\over\displaystyle 1-c_nt- {\strut
b_{n+1}t^2\over\displaystyle \ddots }}}}}, $$
where $c_n=x^nq^n([n]_q+[n+1]_q)$ and $b_{n+1}=x^{2n+1}q^{2n+1}([n+1]_p)^2$
for $n\geq 0$.
 \end{theorem}  
\begin{proof}
    For $\pi\in \SSS_n$, if $\phi_{FZ}(\pi)=(\textbf{s},\textbf{p})$, then we have the following.
    \begin{eqnarray*}
    \dep(\pi)&=&\mathrm{area}(\phi_{FZ}(\pi)),  \hspace{5 mm} \mathrm{area}(\phi_{FZ}(\pi))= \sum_{i=1}^{n} h_i(\textbf{s}),\\
    \inv(\pi)&=&\mathrm{area}(\phi_{FZ}(\pi))+\mathrm{nest}(\pi)
    = \sum_{i=1}^n h_i(\textbf{s}) + \mathrm{nest}(\pi) = \sum_{i=1}^n \big(h_i(\textbf{s})+ p_i \big).
    \end{eqnarray*}
The first two equations are well known in the literature, see 
\cite{petersen-guay-paquet-displacement},\cite{clarke-stein-zeng}. 
The third equation was noted by Viennot, de Medicis \cite{inv-foata-zeil-viennot}.
    \begin{eqnarray*}
       F_n(x,q)&=&\sum_{(\textbf{s},\textbf{p})\in \mathcal{LH}_n^*} x^{\mathrm{area}(\textbf{s},\textbf{p})}q^{\mathrm{area}(\textbf{s},\textbf{p})+\sum_{i=1}^n p_i}\\
       &=&\sum_{(\textbf{s},\textbf{p})\in \mathcal{LH}_n^*} \prod_{s_i=\mathrm  N}\! x^{h_i}q^{h_i+p_i}
\prod_{s_i=\mathrm S}\!x^{h_i}q^{h_i+p_i}\prod_{s_i=\mathrm {dE}}\!x^{h_i}q^{h_i+p_i}
\prod_{s_i=\mathrm E}\!x^{h_i}q^{h_i+p_i}\\
\!&=&\!\sum_{\textbf{s}\in 2-\mathrm{Motz}_n}\prod_{s_i=\mathrm{N}}\!x^{h_i}q^{h_i}[h_i+1]_q
\prod_{s_i=\mathrm S}\!x^{h_i}q^{h_i}[h_i]_q
\prod_{s_i=\mathrm{dE}}\!x^{h_i}q^{h_i}[h_i]_q\prod_{s_i=\mathrm E}\!x^{h_i}q^{h_i}[h_i+1]_q 
    \end{eqnarray*}
where 2-$\mathrm{Motz}_n$ is the set of 2-Motzkin paths of length $n$.  
Therefore, by Flajolet's theory of continued fractions \cite{cont-frac-flaj}, 
we have completed our proof.
\end{proof}

By the bijection of Clarke-Steingr{\'i}mmson-Zeng in \cite{clarke-stein-zeng}, we know 
that $(\dep,\inv)$ is equidistributed with $(\drops,\mad)$. 
Therefore,  we have 
$$F_n(x,q)=\sum_{w\in \SSS_n} x^{\dep(w)}q^{\inv(w)}=
\sum_{w\in \SSS_n}x^{\drops(w)}y^{\mad(w)}.$$

As a corollary, we can enumerate the permutations, with respect to $\dep$ and $\inv$, in the pre-image of Motzkin paths under the $\phi_{\Pet}$ map.       
\begin{corollary}
    \label{cor: kyle-inv}
    Let $p\in \mathrm{Motz}_n$. Then
    \begin{equation}
        \sum_{w\in \phi_{\Pet}^{-1}(p)} q^{\inv(w)}x^{\dep(w)} = \prod_{s_i=\mathrm{N}}\!x^{h_i}q^{h_i}[h_i+1]_q
\prod_{s_i=\mathrm S}\!x^{h_i}q^{h_i}[h_i]_q
\prod_{s_i=\mathrm{E}\text{ or }\mathrm{dE}}\!x^{h_i}q^{h_i}([h_i]_q+[h_i+1]_q) 
    \end{equation}
\end{corollary}

\section{Type-B Coxeter groups}
\label{sec:type-B}

%Let $\BB_n$ be the set of %signed permutations on $[n]$.  That is, these are 
%bijections $\sigma$ from the set
%$ [\pm n] = \{-n,-(n-1),\ldots,-1,1,2,\ldots,n\}$ to itself such that 
%$\sigma(-i) = -\sigma(i)$ for all $i \in [n]$.  %Thus, any such $\sigma$ is 
%defined when given $\sigma(i)$ for $i \in [n]$.  Clearly, 
%$|\BB_n| = 2^n n!$.  
%$\BB_n$ is generated by $s_0, s_1, \ldots, s_{n-1}$ where for $i > 0$,
%$s_i$ is the transposition $(i,i+1)$ as in the type-A case and $s_0 = -1,2,\ldots,n$
%is the transposition that flips the sign of the first element.  
%We cover the combinatorial definitions of a few parameters of $\BB_n$.
For $\sigma = \sigma_1, \sigma_2, \ldots, \sigma_n 
\in \BB_n$, let $\Negs(\sigma) = \{i \in [n]: \sigma_i < 0 \}$ be the set 
of indices where $\sigma$ takes negative values and let 
$\nsum(\sigma) = - ( \sum_{i \in \Negs(\sigma)} \sigma_i)$ be the absolute 
value of the sum of the negative components of $\sigma$.  Recall the 
definition of the 
number of type-A inversions of $\sigma$ as 
 $\inva(\sigma) = |\{ 1 \leq i < j \leq n : \sigma_i > \sigma_j \} |$.
Here, comparison is done with respect to the standard order on $\ZZ$.
Define the number of inversions of $\sigma \in \BB_n$ as 
$\invb(\sigma) = \nsum(\sigma) + \inva(\sigma)$.  The above definition
is due to Brenti \cite{brenti-q-eulerian-94} who also shows that 
$\invb(\sigma)$  %= \ell(\sigma)$.  Here, $\ell(\sigma)$ 
is the length of $\sigma$ in the Coxeter-theoretic sense with 
respect to the generators
$s_0, s_1, \ldots, s_{n-1}$.

Let $\sigma = \sigma_1, \sigma_2, \ldots, \sigma_n \in \BB_n$.  Define
$\sigma_0 = 0$ and define $\DescSetB(\sigma) = \{0 \leq i < n: \sigma_i > \sigma_{i+1} \}$.
This combinatorial definition of descents can be found in \cite{bjorner-brenti}.
Recall that $\drpb(\sigma) = \sum_{i \in \DescSetB(\sigma)} (\sigma_i - \sigma_{i+1})$.
If $\sigma \in \BB_n \cap \SSS_n$, then clearly $\drpb(\sigma) = \drp(\sigma)$.
For $n \geq 1$, define 
$\displaystyle \SgnDrpB_n(q) = \sum_{\sigma \in \BB_n} (-1)^{\invb(\sigma)} q^{\drpb(\sigma)}.$  
Our main result is the following type-B counterpart of Theorem \ref{thm:sgn_drp_univ}.  

\begin{theorem}
  \label{thm:sgn_drp_typeB}
	For $n \geq 1$, $\SgnDrpB_n(q) = (1-q)^n$.
\end{theorem}

We begin with a few lemmas.  The following is a type-B counterpart of Lemma 
\ref{lem:increase_typeA}.  As its proof is similar to the proof of 
Lemma \ref{lem:increase_typeA},  we omit a proof.  The only 
difference is that we only claim similar results for the intermediate permutations 
$w_{n+1}, w_n \ldots w_2$ and not $w_1$ as $r_1$ could be $s_0$ which flips the sign
of the first element.

\begin{lemma}
  \label{lem:increase_typeB}
Let $W = \BB_n$ %be a rank $n$, type-B Coxeter group 
and let $w \in W$ have intermediate 
elements $w_i$, where $1 \leq i \leq n+1$.  For $2 \leq i \leq n+1$, let the 
one-line notation of $w_i$ be $w_i = a_0a_1,a_2,\ldots,a_n$, where $a_0 = 0$.
Then, $a_0 < a_1 < a_2 \cdots < a_{i-1}$.  Let $\rd(w) = [r_n][r_{n-1}]\cdots[r_1]$
and let $w_1 = b_0 b_1, b_2, \ldots, b_n$ where $b_0 = 0$.  If 
$r_1 \not= s_0$, then $b_0 < b_1$.
\end{lemma}

The following is a type-B counterpart of Lemma 
\ref{lem:drp_preserve_typeA}.

\begin{lemma}
  \label{lem:drp_preserve_typeB}
 % Let $W = \BB_n$.  be the rank $n$ type-B Coxeter group.  
  For $\sigma \in \BB_n$, $\drpb(\sigma) = \drpb(g_B(\sigma))$.
\end{lemma}
\begin{proof}
 Let $T_n = \{ \psi \in \BB_n: g_B(\psi) = \psi\}$.  The lemma is
trivial if $\sigma \in T_n$ and hence we assume that $\sigma \not\in T_n$.
We break the proof into two cases depending on whether 
$\rd(\sigma)$ has $\nml$ factors.  
 
{\bf Case 1:} (When $\rd(\sigma)$ has no $\nml$ factors.)  
The argument is very similar to the type-A argument.  We start from
$w_{n+1} = \id$ and note that at the $i$-th step, the $(n-i)$th element 
from the left gets into its correct place with its correct sign.  After
we perform $w_{i+1} r_n$, we get $w_n$ and let $x_i$ be the $(n-i)$th
element from the left in $w_i$ (that is the element which came 
into its correct place in $w_i$).  We might get into trouble only when 
$x_i$ came from lesser than two positions to the left of $(n-i)$ AND
if $x_i$ is negative.  Otherwise, it can be seen that \eqref{eqn:trio}
holds and we are done.

Thus, for all $i \in [n]$, we have $r_i \not= u_i$ and $r_i \not= v_i$.  
Let $\psi = g_B(\sigma)$ 
and let $t$ be the smallest index such that $\ell(r_t) \geq 2$. 
Recall the intermediate permutations $\sigma_i$ and $\psi_i$ for
$1 \leq i \leq n+1$.  Let $\psi_i = \psi_{i+1} q_i$.
%An argument similar to that employed in the proof of Theorem \ref{thm:sgn_drp_univ} will be used to 
As in the earlier proof, we will show that $\drpb(\sigma_i) = \drpb(\psi_i)$ for 
$1 \leq i \leq n+1$.  For $t \leq i \leq n+1$, we have 
$\sigma_i = \psi_i$ and hence $\drpb(\sigma_i) = \drpb(\psi_i)$.
Recall that $\sigma_t = \sigma_{t+1} r_t$, that $\psi_t = \psi_{t+1} r_t s_{t-2}$
and that $\ell(r_j) \leq 1$ for $j < t$.  
After using the relation $s_i^2 = \id$, we see that either $r_{t-1} = 1$ and
$q_{t-1} = s_{t-2}$ or $r_{t-1}  = s_{t-2}$ and $q_{t-1} = 1$.  Assume that
$r_{t-1} = s_{t-2}$ and hence $\sigma_t = \psi_t = \psi_{t-1}$. 
Further, $r_t \in W^{\bkt{t}}$ and $r_t \not= u_t, v_t$.  Thus, 
$\drpb(\sigma_{t-1}) = \drpb(\psi_{t-1})$.  A similar argument shows that
$\drpb(\sigma_i) = \drpb( \psi_i)$ for $1 \leq i < t$.

{\bf Case 2:} (When $\rd(\sigma)$ has $\nml$ factors.)  We carry over
our notation from Case 1.  Let $t$ be the
largest index with $r_t$ being $\nml$ and without loss of generality 
let $r_t = u_t$ and $q_t = v_t$.  Let $w_{t+1} \cdot u_t = w_t$ and
$z_{t+1} v_t = z_t$.  Note that $n \geq t \geq 2$.

We first show that $\drpb(w_t) = \drpb(z_t)$.  We clearly have 
$w_{t+1} = z_{t+1}$.  Let $w_{t+1} = z_{t+1} = x_1, x_2, \ldots, x_n$.
Since $w_t = w_{t+1} u_t$ and $z_{t+1} = z_t v_t$, we have

%\begin{eqnarray*}
$w_t   =  x_1, x_2 \ldots, x_{t-1},\ol{x_{t}}, x_{t+1}, \ldots, x_n \mbox{ and }
z_t   =  x_1, x_2 \ldots, x_{t},\ol{x_{t-1}}, x_{t+2}, \ldots, x_n.$
%\end{eqnarray*}

{We clearly have that the $\DescSetB(w_t)=\DescSetB(w_{t+1})\cup \{t\}$ and $\DescSetB(z_t)=\DescSetB(z_{t+1}) \cup \{t\}$. Thus, 
we get $\drpb(w_{t})=\drpb(w_{t+1})+x_{t-1}+x_{t}$ and $\drpb(z_{t})=\drpb(z_{t+1})+x_{t-1}+x_{t}$. As $w_{t+1} = z_{t+1}$, we
get that $\drpb(w_t) = \drpb(z_t)$ completing the proof.}
We next show that $\drpb(w_{t-1})=\drpb(z_{t-1})$.  When $k\ge 1$, 
we have $r_{t-k}=q_{t-k}$.  Without loss of generality, assume that 
$r_{t-1}\neq 1$. Therefore, we have
$w_{t-1}=x_1,x_2,\ldots,x_{t-1},\alpha, \ol{x_t}, \ldots, x_n$  and 
$z_{t-1}=x_1,x_2,\ldots,x_{t},\alpha, \ol{x_{t-1}}, \ldots, x_n$,
where $\alpha=y_i>0$ or $\alpha=\ol{y_i}$ for some $1\le i <t$.

{\bf Subcase (a) (When $\alpha=y_i$)}  Since $y_i < x_{t-1} < x_t$, we have a 
descent at position $t-2$.  Let $D= \sum_{i\in \DescSetB(w_t) \cap [t+1,n-1]} 
(x_i-x_{i+1})$. Then, we have 
%\begin{equation*}
    $\drpb(w_{t-1})=(x_{t-1}-\alpha)+(\alpha+x_t)+D$  and 
    $\drpb(z_{t-1})=(x_{t-1}-\alpha)+(\alpha+x_{t-1})+D.$ %\end{equation*}
Clearly, both are equal.

{\bf Subcase (b) (When $\alpha=\ol{y_i}$)} We have $\ol{y_i}>\ol{x_{t-1}}>\ol{x_t}$ 
and therefore
%\begin{eqnarray*}
    $\drpb(w_{t-1})=(x_{t-1}+y_i)+(x_t-y_i)+D$ and 
    $\drpb(z_{t-1})=(x_{t}+y_i)+(x_{t-1}-y_i)+D$.
%\end{eqnarray*}
Clearly,  both are equal.
An equivalent argument can be used to show that $\drpb(w_{t-k})=\drpb(z_{t-k})$ 
when $k\ge 2$ and is therefore, omitted.  The proof is complete.
\end{proof}
%\begin{eqnarray*}
%w_t  & = & x_1, x_2 \ldots, x_{n-t-1},\ol{x_{n-t}}, x_{n-t+1}, \ldots, x_n \mbox{ and } \\
%z_t  & = & x_1, x_2 \ldots, x_{n-t},\ol{x_{n-t-1}}, x_{n-t+1}, \ldots, x_n.
%\end{eqnarray*}

%{\bf Do we need some justification? We have $\ol{x_{n-t}} < x_{n-t+1}$ and 
%$\ol{x_{n-t-1}} < x_{n-t+1}$.  We further have $x_{n-t}>0$ and $x_{n-t-1} > 0$.}
%We clearly have 
%$\DescSet(w_t) = \DescSet(w_{t+1}) \cup \{n-t-1\}$ and 
%$\DescSet(z_t) = \DescSet(z_{t+1}) \cup \{n-t-1\}$.
%Thus, we get $\drpb(w_t) = \drpb(w_{t+1}) + x_{n-t} + x_{n-t-1}$ and 
%$\drpb(z_t) =  \drpb(z_{t+1}) + x_{n-t} + x_{n-t-1}$.  As $w_{t+1} = z_{t+1}$, we
%get that $\drpb(w_t) = \drpb(z_t)$ completing the proof.
%}

%We next show that $\drpb(w_{t-k}) = \drpb(z_{t-k})$ when $k \geq 1$.
%When $k \geq 1$, note that $r_{t-k} = q_{t-k}$ and thus we have 
%$w_{t-k+1} r_{t-k} = w_{t-k}$ and $z_{t-k+1} r_{t-k} = z_{t-k}$.
%Assume that $r_{t-k} \not= 1$.  If $w_{t-k} 
%= x_1, x_2, \ldots, x_{n-(t-k)},\ldots,x_n$ \green{(not to be confused with the $x_i$ from earlier?)}, our argument 
%will depend on the sign of the element 
%$x_{n-(t-k)}$ in position $n- (t-k)$. 

%{\bf Subcase (a): when $x_{n-(t-k)}$ is positive:}

%{\bf Subcase (b): when $x_{n-(t-k)}$ is negative:}

With these lemmas in place, we move on to the proof of 
Theorem \ref{thm:sgn_drp_typeB}.

\begin{proof}(Of Theorem \ref{thm:sgn_drp_typeB})
  %As seen, $W = \BB_n$ has rank $n$.  
  Define $T_n$ to be the set of $\sigma \in \BB_n$ such that $\ird(\sigma)$ has no
  ascents.  These are permutations whose canonical reduced 
  word is of the form $s_{i_1}s_{i_2} \cdots s_{i_k}$ where $n > i_1 > i_2
  > \cdots > i_k \geq 0$.  As there are $n$ generators, such 
  permutations are clearly in bijection with subsets $L$ of $[n-1]_0$
  and thus $|T_n| = 2^n$.  Define $K_n = \BB_n - T_n$.  Consider 
  the involution $f$ defined in Section \ref{sec:invol}.  By 
  Lemmas \ref{lem:sign_rev_b}
  and \ref{lem:drp_preserve_typeB},
  $\sum_{\sigma \in K_n} (-1)^{\invb(\sigma)} q^{\drpb(\sigma)} = 0$ and thus
  $\sum_{\pi \in \BB_n} (-1)^{\invb(\sigma)} q^{\drpb(\sigma)} = 
  \sum_{\pi \in T_n} (-1)^{\invb(\sigma)} q^{\drpb(\sigma)}$.

  For $L \subseteq [n-1]_0$, $L = \{i_1, i_2, \ldots, i_k \}$, we write $\langle L \rangle$ 
  or $\{i_k, i_{k-1}, \ldots, i_1 \}_>$ for the sequence of elements of $L$ arranged 
  in decreasing 
  order.  We write $v = \prod_{i \in \langle L \rangle} s_i$ to denote the signed 
  permutation 
  $\sigma \in T_n$ whose canonical reduced word is $v$.  Clearly, the signed permutation 
  $\sigma$ with 
  canonical reduced word $\prod_{i \in \langle L \rangle} s_i$ has $\invb(\sigma) = 
  (-1)^{|L|}$.
  Further, if we cluster the elements of $L$ into consecutive elements, then it is easy to see
  that $\drpb(\sigma) = |L|$.  Thus,
  $$ \sum_{\sigma \in T_n} (-1)^{\invb(\sigma)} q^{\drpb(\sigma)} = \sum_{L \subseteq [n]} (-1)^{|L|} q^{|L|} 
  = (1-q)^n.$$
  The proof is complete.
\end{proof}

\section{Type D Coxeter groups}
\label{sec:type-d}

Recall that $\DD_n$ the type-D Weyl group, is the subset of signed permutations 
$\sigma \in \BB_n$ such that $|\Negs(\sigma)| \equiv 0$ (mod 2).  That is, 
it is the set of signed permutations with
an even number of negative elements. %in $\{ \sigma(i): i \in [n] \}$. 
For $\sigma \in \DD_n$, Brenti \cite{brenti-q-eulerian-94} showed that 
its number
of inversions is given by $\invd(\sigma) = \invb(\sigma) - |\Negs(\sigma)|$.
When $n \geq 2$, for $\sigma = \sigma_1, \sigma_2, \ldots, \sigma_n 
\in \DD_n$, define $\sigma_0 = -\sigma_2$.  Define 
$\DescSetD(\sigma) = \{0 \leq i < n: \sigma_i > \sigma_{i+1} \}$.
Define $\drops_D(\sigma) = \sum_{i \in \DescSetD(\sigma)} 
(\sigma_{i+1} - \sigma_i)$ and 
%let 
%$$\SgnDrpD_n(q) = \sum_{\sigma \in \DD_n} (-1)^{\invd(\sigma)} q^{\drpd(\sigma)}.$$
define $\Sdropd_n(q)=\sum\limits_{\pi 
\in \DD_n} (-1)^{\invd(\pi)}q^{\drops_D(\pi)}.$
The main result of this Section is the following.
\begin{theorem}
\label{thm-typed}
     For $n\ge 2$, $\Sdropd_n(q)=(1-q^3)(1-q)^{n-1}$.
\end{theorem}

%To prove our main result, we need a definition, a technical lemma and a proposition.
First, note that the combinatorial definition of $\invd$ can be extended to 
any signed permutation in $\BB_n$ without any changes.
Define  $\zdrops(\pi)=\left\{
\begin{array}{ll}
\drops_B(\pi)  & \mbox{if } \pi_1 > 0, \\
\drops_B(\pi)-\pi_1 & \mbox{if } \pi_1 <0.
\end{array}
\right.$

\begin{lemma} \label{thm-zerolemma}
When $n\ge 2$, we have $\sum\limits_{\pi \in \BB_n-\DD_n} (-1)^{\invd(\pi)}q^{\zdrops(\pi)}
=0.$  Similarly, when $n \ge 2$, we have $\sum\limits_{\pi \in \DD_n} 
(-1)^{\invd(\pi)}q^{\zdrops(\pi)}=0.$
\end{lemma}
\begin{proof}
We only prove the first identity and omit a proof of the second as the arguments are 
very similar.  
Let $A_1$ be the elements of $\BB_n-\DD_n$ that have $n-1$ or $\overline{n-1}$ 
occurring before $n$ or $\overline{n}$ and
let $A_2$ be the elements of $\BB_n - \DD_n$ in which $n$ or $\overline{n}$ 
occurs before $n-1$ or $\overline{n-1}$.
It is clear that $A_1 \sqcup A_2$ is a partition of $\BB_n-\DD_n$ and $|A_1|=|A_2|$. 
We define a bijection $g: A_1 \mapsto A_2$. 
%that takes permutations in $A_1$ to permutations in $A_2$.  
We split into cases here to describe our map $g$:  If 
\begin{enumerate}
\item If $\pi = \sigma_1,n-1,\sigma_2,n,\sigma_3$ where $\sigma_i$, for $i=1,2,3$
are contiguous subwords, then set $g(\pi)=\sigma_1,n,\sigma_2,n-1,\sigma_3$.
In this case, one can check that $\invd(g(\pi))=\invd(\pi)+1$ and $\zdrops(g(\pi)) =
\left\{
\begin{array}{ll}
\zdrops(\pi)  & \mbox{if } \sigma_3 \neq \phi, \\
\zdrops(\pi)+1 & \mbox{if } \sigma_3 = \phi.
\end{array}
\right.$
\item  If $\pi = \sigma_1,\overline{n-1},\sigma_2,n,\sigma_3$, then set $g(\pi)
=\sigma_1,\overline{n},\sigma_2,n-1,\sigma_3$. In this case, it can be seen that 
$\invd(g(\pi))=\invd(\pi)+1$ and $\zdrops(g(\pi)) =
\left\{
\begin{array}{ll}
\zdrops(\pi)  & \mbox{if } \sigma_3 \neq \phi, \\
\zdrops(\pi)+1 & \mbox{if } \sigma_3 = \phi.
\end{array}
\right.$
\item If $\pi = \sigma_1,n-1,\sigma_2,\overline{n},\sigma_3$, then set 
$g(\pi)=\sigma_1,n,\sigma_2,\overline{n-1},\sigma_3$. In this case, it can be seen that
$\invd(g(\pi))=\invd(\pi)-1$ and  $\zdrops(g(\pi)) =\zdrops(\pi).$
\item If $\pi = \sigma_1,\overline{n-1},\sigma_2,\overline{n},\sigma_3$, then set 
$g(\pi)=\sigma_1,\overline{n},\sigma_2,\overline{n-1},\sigma_3$. In this case, 
we have $\invd(g(\pi))=\invd(\pi)+1$ and $\zdrops(g(\pi)) =
\left\{
\begin{array}{ll}
\zdrops(\pi)  & \mbox{if either }  \sigma_1  \mbox{or } \sigma_2 \neq \phi, \\
\zdrops(\pi)+1 & \mbox{if } \sigma_1, \sigma_2  = \phi.
\end{array}
\right.$
\end{enumerate}
Using $g$, we get $$\sum\limits_{\pi \in \BB_n-\DD_n} (-1)^{\invd(\pi)}
q^{\zdrops(\pi)}=\sum\limits_{w \in A_1} \Big[(-1)^{\invd(w)}q^{\zdrops(w)}+ 
(-1)^{\invd(g(w))}q^{\zdrops(g(w))}\Big].$$

On the right hand side, the only $w$ for which the two terms do not cancel out 
are when the last letter is $n$ or $n-1$ (by induction, the contribution of 
such permutations is clearly $(1-q)\Bigg( \sum\limits_{\pi \in \BB_{n-1}-\DD_{n-1}} 
 (-1)^{\invd(\pi)}q^{\zdrops(\pi)}\Bigg)$) and when the first two letters 
are $\overline{n-1},\overline{n}$ (again, inductively, such permutations
contribute $\sum\limits_{\pi \in \BB_{n-2}-\DD_{n-2}} 
(-1)^{\invd(\pi)}q^{\zdrops(\pi)}$).  Thus, we get 
\begin{eqnarray*}
& & \sum\limits_{\pi \in \BB_n-\DD_n} 
(-1)^{\invd(\pi)}q^{\zdrops(\pi)}\\
& & = (1-q)\sum\limits_{\pi \in \BB_{n-1}-\DD_{n-1}} (-1)^{\invd(\pi)}q^{\zdrops(\pi)}
+(q^{2n}-q^{2n-1})\sum\limits_{\pi \in \BB_{n-2}-\DD_{n-2}} 
(-1)^{\invd(\pi)}q^{\zdrops(\pi)}.
\end{eqnarray*}
When $n=2$, %the permutations in $\BB_2-\DD_2$ are $\overline{1}2,\overline{2}1,1\overline{2},2\overline{1}$ 
it is easy to see that $\sum\limits_{\pi \in \BB_2-\DD_2} (-1)^{\invd(\pi)}q^{\zdrops(\pi)}
%=1-1+q^3-q^3
=0.$
%The proposition for $n=3$ has been also checked using Wolfram-Alpha.
This completes the proof.
\end{proof}

\begin{lemma}
\label{lem:induct_type_d}
For natural numbers $n\ge 4$, we have 
    $\Sdropd_n(q)=(1-q)\Sdropd_{n-1}(q)$.
\end{lemma}
\begin{proof}
Define $A_1$ to be the subset of $\DD_n$ having permutations with $n-1$ or 
$\overline{n-1}$ occurring before $n$ or $\overline{n}$ and $A_2=\DD_n-A_1$.  
It is straightforward to see that these sets have the same cardinality. 
We describe a bijection $h: A_1 \mapsto A_2$. The map on $\DD_n$ that acts 
by $h$ on $A_1$ and $h^{-1}$ on $A_2$ is an involution.
We describe how to map these signed permutations after splitting them into
the following cases.
%For elements of $A_1$, we can factorise the signed permutation into one of the following:
%If 
%\begin{enumerate}
%\item $\pi=\sigma_1,n-1,\sigma_2,n,\sigma_3$
%\item $\pi=\sigma_1,\overline{n-1},\sigma_2,n,\sigma_3$
%\item $\pi=\sigma_1,n-1,\sigma_2,\overline{n},\sigma_3$
%\item $\pi=\sigma_1,\overline{n-1},\sigma_2,\overline{n},\sigma_3$
%\end{enumerate}
%where $\sigma_1,\sigma_2,\sigma_3$ are contiguous subwords of $\pi$ (not necessarily non-empty).\\
     
\begin{enumerate}
\item  If $\pi = \sigma_1,(n-1),\sigma_2,n,\sigma_3$, then set $h(\pi)=\sigma_1,n,
\sigma_2,n-1,\sigma_3$.  One can check that $\invd(h(\pi))=\invd(\pi)+1$ and 
    $\drops_D(g(\pi)) =
\left\{
\begin{array}{ll}
\drops_D(\pi)  & \mbox{if } \sigma_3 \neq \phi, \\
\drops_D(\pi)+1 & \mbox{if } \sigma_3 = \phi.
\end{array}
\right.$
\item  If $\pi = \sigma_1,\overline{n-1},\sigma_2,n,\sigma_3$, then 
set $h(\pi)=\sigma_1,\overline{n},\sigma_2,n-1,\sigma_3$.  One can
check that $\invd(h(\pi))=\invd(\pi)+1$ and $\drops_D(h(\pi)) =
\left\{
\begin{array}{ll}
\drops_D(\pi)  & \mbox{if } \sigma_3 \neq \phi, \\
\drops_D(\pi)+1 & \mbox{if } \sigma_3 = \phi.
\end{array}
\right.$
\item  If $\pi = \sigma_1,n-1,\sigma_2,\overline{n},\sigma_3$, then 
set $h(\pi)=\sigma_1,n,\sigma_2,\overline{n-1},\sigma_3$.  We have 
$\invd(h(\pi))=\invd(\pi)-1$ and $\drops_D(h(\pi))=\left\{
\begin{array}{ll}
\drops_D(\pi)  & \mbox{if } \sigma_1,\sigma_2 \neq \phi, \\
\drops_D(\pi)+1 & \mbox{if } \sigma_1,\sigma_2 = \phi.
\end{array}
\right.$
\item  If $\pi = \sigma_1,\overline{n-1},\sigma_2,\overline{n},\sigma_3$, 
then set $h(\pi)=\sigma_1,\overline{n},\sigma_2,\overline{n-1},\sigma_3$.
We have $\invd(h(\pi))=\invd(\pi)+1$ and $\drops_D(h(\pi))=\left\{
\begin{array}{ll}
\drops_D(\pi)  & \mbox{if } \sigma_1,\sigma_2 \neq \phi, \\
\drops_D(\pi)+1 & \mbox{if } \sigma_1,\sigma_2 = \phi.
\end{array}
\right.$
\end{enumerate}
Using $h$, we get
\begin{eqnarray*}
\Sdropd_n(q)&=&(1-q)\Sdropd_{n-1}(q)\\
& & + \sum\limits_{\lbrace \pi \in A_1:\ \pi_1=n-1,
\pi_2=\overline{n}\rbrace} \bigg( (-1)^{\invd(\pi)}q^{\drops_D(\pi)} 
+ (-1)^{\invd(h(\pi))}q^{\drops_D(h(\pi))}\bigg) \\
& & +\sum\limits_{\lbrace \pi \in A_1:\ \pi_1=\overline{n-1},
\pi_2=\overline{n}\rbrace} \bigg( (-1)^{\invd(\pi)}q^{\drops_D(\pi)} 
+ (-1)^{\invd(h(\pi))}q^{\drops_D(h(\pi))}\bigg)
\end{eqnarray*}

We show that both the second and third terms of the sum above are zero. 
If $\pi_1=n-1,\pi_2=\overline{n}$, then clearly,  $\pi_1 = n-1,\overline{n},
\sigma$ where $\sigma\in \BB_{n-2}-\DD_{n-2}$.  Thus,
\begin{eqnarray*}
&&\sum\limits_{\lbrace \pi \in A_1:\ \pi_1=n-1,\pi_2=\overline{n}\rbrace} 
\bigg( (-1)^{\invd(\pi)}q^{\drops_D(\pi)} + (-1)^{\invd(h(\pi))}
q^{\drops_D(h(\pi))}\bigg)\\
&=& \sum\limits_{\sigma\in \BB_{n-2}-\DD_{n-2}} \bigg( (-1)^{\invd(\pi)}
q^{2n+\zdrops(\sigma)} + (-1)^{\invd(g(\pi))}q^{2n-1+\zdrops(\sigma)}\bigg)\\
&=&((-1)^{2n-2}q^{2n}+(-1)^{2n-3}q^{2n-1})
\underbrace{\bigg(\sum\limits_{\sigma \in \BB_{n-2}-\DD_{n-2}} 
(-1)^{\invd(\sigma)}q^{\zdrops(\sigma)}\bigg)}_{=0\ 
(\mbox{from \eqref{thm-zerolemma}})}
\end{eqnarray*}
Similarly, if $\pi_1=\overline{n-1},\pi_2=\overline{n}$, then we have 
$\pi_1 = \overline{n-1},\overline{n},\sigma$ where $\sigma\in \DD_{n-2}$.  
Hence,
\begin{eqnarray*}
& &\sum\limits_{\lbrace \pi \in A_1:\ \pi_1=\overline{n-1},\pi_2=\overline{n}\rbrace} 
\bigg( (-1)^{\invd(\pi)}q^{\drops_D(\pi)} + (-1)^{\invd(h(\pi))}
q^{\drops_D(h(\pi))}\bigg)\\
&=& \sum\limits_{\sigma\in \DD_{n-2}} \bigg( (-1)^{\invd(\pi)}
q^{2n+\zdrops(\sigma)} + (-1)^{\invd(h(\pi))}q^{2n-1+\zdrops(\sigma)}\bigg)\\
&=&((-1)^{2n-2}q^{2n}+(-1)^{2n-3}q^{2n-1})\underbrace{\bigg(
\sum\limits_{\sigma \in \DD_{n-2}} (-1)^{\invd(\sigma)}q^{\zdrops(\sigma)}
\bigg)}_{=0\ (\mbox{from \eqref{thm-zerolemma}})}
\end{eqnarray*}
Our proof is complete.
    \end{proof}
\begin{proof}[Proof of Theorem \eqref{thm-typed}]
It is easy to see that $\Sdropd_3(q)=(1-q^3)(1-q)^2$.  Applying
Lemma \ref{lem:induct_type_d} completes the proof.
\end{proof}

\section{Two applications of our results}
We cover two applications of our results in this section.

\subsection{Mean and Variance}
\label{subsec:asymp_norm}

Graham and Diaconis showed that when one samples permutations from $\SSS_n$ 
at random, then the distribution of its Spearman's measure $D(\pi)$ 
is asymptotically normal.  Their result is the following.

\begin{theorem}[Graham and Diaconis]
\label{thm:mean_var_D-perm}
Let $\pi$ be a permutation chosen independently and uniformly in $\SSS_n$.  Then,
as $n \to \infty$, 
\begin{equation}
\label{eqn:mean_var_in_Sn}
E(D(\pi)) = \frac{1}{3}n^2 + O(n) \mbox{ and } \var(D(\pi)) = \frac{2}{45}n^3 + O(n^2)
\end{equation}
and $D(\pi)$ standardized by \eqref{eqn:mean_var_in_Sn} is asymptotically normal.
\end{theorem}

We need the following lemma  which is implicit in the  second proof of
\cite[Theorem 1.2]{fulman-kim-lee-petersen_joint-distrib-descents-sign}.
For positive integers $n$, let 
$F_n(t)= \sum_{k=0}^n f_{n,k}t^k$ and $G_n(t)=\sum_{k=0}^n g_{n,k}t^k$ 
be sequences of polynomials with $f_{n,k}, g_{n,k} \geq 0$ for all $k$.
Further, for all positive integers $n$, let $F_n(1) > 0$ and $G_n(1) > 0$.
We will consider the random variable $X_f^n$ (and $X_g^n$) which takes 
the value $k$ with probability $ \displaystyle \frac{f_{n,k}}{F_n(1)}$ 
\Big(and
$\displaystyle \frac{g_{n,k}}{G_n(1)}$ respectively\Big).  The following
form of the Lemma appears in \cite[Lemma 8]{dey-siva-eulerian-clt-carlitz-ids}
and gives a condition for moments of $X_f^n$ and $X_g^n$ to be identical.

\begin{lemma}
\label{lem:connection_between_all_and_plus}
Let $F_n(t)$ and $G_n(t)$ be as described above.  Suppose we have 
a sequence of polynomials $H_n(t)$, a non-zero real number $\lambda$ and a 
sequence of positive integers $\ell_n$ such that 
\begin{equation}
\label{eqn:diff_by_1-t-power}
F_n(t)= \lambda G_n(t) \pm (1-t)^{\ell_n} H_n(t).
\end{equation}

Then, for $r<\ell_n$, the $r$-th moment of $X_f^n$
is identical to the $r$-th moment of $X_g^n$.
\end{lemma}

%When a signed enumeration result is in the form as given in 
% Lemma \ref{lem:connection_between_all_and_plus}, we get results about 
%the mean and variance of the variable that is enumerated.

As a corollary of Theorem \ref{thm:sgn_drp_univ}, Theorem \ref{thm:mean_var_D-perm}
and Lemma \ref{lem:connection_between_all_and_plus}, we get the following.  A 
scaling is necessary as for all $\pi \in \SSS_n$, we have 
$2 \cdot \drp(\pi) = f(D(\pi))$ where $f$ is the map used by Reifegerste in 
\cite{reifegerste-bi-incr}.

\begin{theorem}
\label{thm:mean_var_drp-even-perm}
Let $\pi$ be a permutation chosen independently and uniformly in $\AAA_n$.  Then,
as $n \to \infty$, 
\begin{equation}
\label{eqn:mean_var_in_An}
E(\drp(\pi)) = \frac{1}{2} \cdot \frac{1}{3}n^2 + O(n) \mbox{ and } \var(\drp(\pi)) = \frac{1}{4} \cdot \frac{2}{45}n^3 + O(n^2).
\end{equation}
\end{theorem}

\subsection{A complete matching on the Bruhat order}
\label{subsec:matching}

In this section, we show that our sign-reversing involution defined in Section \ref{sec:invol}
gives a matching in the Bruhat order on $\SSS_n$ and $\BB_n$.  The Bruhat order is a partial order on
permutations $\pi \in \SSS_n$.  There are several definitions of this partial order and for 
our purposes, the most suitable is the following.  Let $\pi, \psi \in \SSS_n$.  Define
$\psi \leq \pi$ if for some reduced word $s_{i_1} s_{i_2} \cdots s_{i_k}$ of $\pi$, there 
exists a subset $S \subseteq [k]$ such that $\prod_{j \in S} s_{i_j}$ is a reduced word
for $\psi$.

%The existence of such a matching gives us an alternate proof of Verma's Theorem 
%(see \cite{verma-mobius-bruhat}) that for
%the Bruhat order on $\SSS_n$, the mobius function $\mu(\id, w) = (-1)^{n \choose 2}$.

\begin{figure}[h]
  \centerline{\includegraphics[scale=0.45]{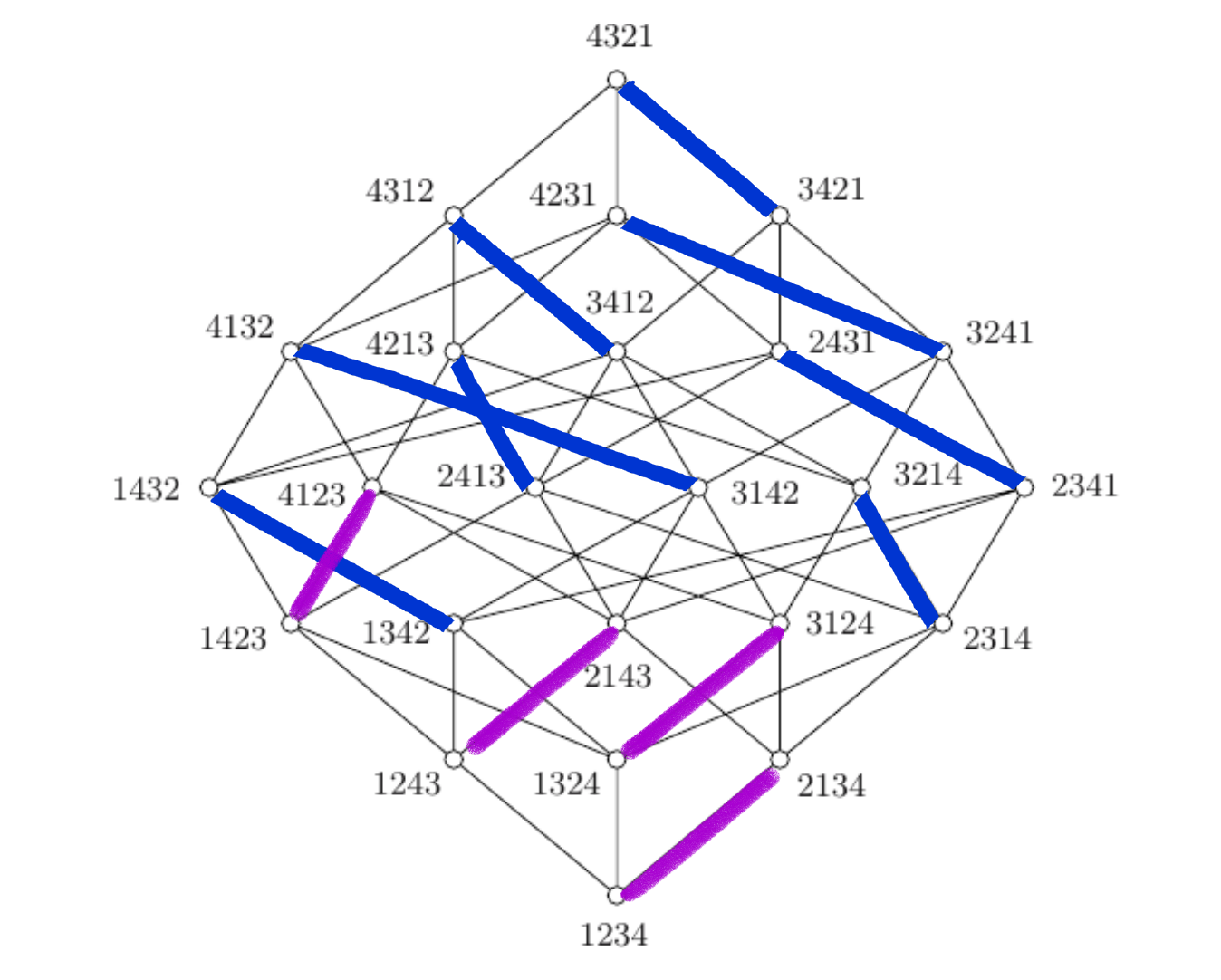}}
  \caption{The Bruhat order on $\SSS_4$, taken from page 31 of Bjorner and Brenti, \cite{bjorner-brenti}.}
  \label{fig:matching_bruhat}
\end{figure}

Let $\pi \in \SSS_n - T_n$ and let $f(\pi)$ be its image.  Since $\pi$ is 
an involution, assume $\inv(\pi) < \inv(f(\pi))$ (otherwise exchange the
roles of $\pi$ and $f(\pi)$).  By Lemma 
\ref{lem:canonical_preserve}, we have $\inv(\pi) = \inv(f(\pi)) - 1$.
Further, by Lemma \ref{lem:canonical_preserve} and 
the definition of the Bruhat order $\pi \leq f(\pi)$ as the canonical
reduced word for $\pi$ is a subword of the canonical reduced word
for $f(\pi)$.

Thus, we get a matching in the Bruhat order on elements of $\SSS_n - T_n$.
We have seen that the elements of $T_n$ are in bijection with subsets 
of $[n-1]$ where simple matchings exist (say between subsets {\it with}
and {\it without} the element $s_1$ in the reduced word).  This gives a matching on the 
set $T_n$ and combining these two, we get a matching on the Bruhat
order of $\SSS_n$.

We illustrate the matching for $\SSS_4$ in Figure \ref{fig:matching_bruhat}.
Both end-points of blue edges (seen better on a colour monitor)  
are elements of $\SSS_n - T_n$ while both end-points of purple edges
are elements in $T_n$.  A file with the canonical reduced word for 
each permutation in $\SSS_4$ and the matching is kept at the url 
http://www.math.iitb.ac.in/$\sim$krishnan/4.bijection-output

Recently Jones \cite{jones-matching-bruhat} gave an explicit matching 
on arbitrary intervals $[u,v]$ in the Bruhat order where $u \leq v$.
His matching depends on a reduced word for the element $v$ and when
specialised to the whole Bruhat order on $\SSS_n$ (i.e. the interval
$[\id,w]$ where $w$ is the unique permutation with largest $\inv$
value) is seen to be different from the one just described.

%A by-product of the involution $f$ is a matching in the Bruhat order of $\SSS_n$.  We illustrate it
%in this extended abstract for $n=4$ below.  

\section*{Acknowledgement}
The first author gratefully acknowledges support from the National Board 
for Higher Mathematics, India.

%\bibliographystyle{acm}
%\bibliography{recent}
%\bibliography{main}

\end{document}